\documentclass[a4paper,12pt,thmsa]{amsart}
\usepackage[a4paper,marginratio={1:1},scale={0.72,0.74},footskip=7mm,headsep=10mm]{geometry}

\usepackage{amsfonts}
\usepackage{amssymb,amsmath,latexsym}
\usepackage[dvips]{graphics}
\usepackage{graphicx,subfigure}
\usepackage[dvips]{color}
\usepackage[T1]{fontenc}
\usepackage[active]{srcltx}
\usepackage{amsmath}
\usepackage{amsfonts}
\usepackage{amssymb}
\usepackage{psfrag}
\usepackage{color}
\usepackage{url}
\usepackage{amsthm}
\usepackage{array}

\usepackage{pst-tree}
\usepackage{lscape}
\usepackage[T1]{fontenc}
\usepackage{pstricks,pstricks-add}

\newtheorem{thm}{Theorem}[section]
\newtheorem{proposition}[thm]{Proposition}
\newtheorem{remark}[thm]{Remark}
\newtheorem{definition}[thm]{Definition}
\newtheorem{corollary}[thm]{Corollary}

\newtheorem{lemma}[thm]{Lemma}

\newcommand{\ed}{\stackrel{\mbox{\tiny $(law)$}}{=}}

\def\qed{\hfill $\square$ \bigskip}

\def\beq{\begin{equation}}               
\def\eeq{\end{equation}}                 
\def\bea{\begin{eqnarray}}             
\def\eea{\end{eqnarray}}               
\def\be*{\begin{eqnarray*}}             
\def\ee*{\end{eqnarray*}}               
\def\ba{\begin{array}}                  
\def\ea{\end{array}}                    

\def\beqlb{\begin{eqnarray}} \def\eeqlb{\end{eqnarray}}
\def\beqnn{\begin{eqnarray*}} \def\eeqnn{\end{eqnarray*}}

\def\<{\langle}  \def\>{\rangle}

\def\bde{\begin{definition}}
\def\ede{\end{definition}}

\def\bth{\begin{thm}}
\def\eth{\end{thm}}
\def\bpr{\begin{proposition}}
\def\epr{\end{proposition}}
\def\ble{\begin{lemma}}
\def\ele{\end{lemma}}
\def\bcor{\begin{corollary}}
\def\ecor{\end{corollary}}
\def\bre{\begin{remark}}
\def\ere{\end{remark}}

\def\p{\mathbb{P}}
\def\ee{\varepsilon}
\def\qed{{\hfill $\Box$ \bigskip}}

\def\qed{{\hfill $\Box$ \bigskip}}

\title[Coding multitype forests: application to the law of the total population]
{Coding multitype forests: application to the law of the total population of branching forests}

\author{Lo\"ic Chaumont}

\address{Lo\"ic Chaumont -- LAREMA -- UMR CNRS 6093, Universit\'e d'Angers, 2 bd Lavoisier, 49045 Angers cedex~01}

\email{loic.chaumont@univ-angers.fr}

\author{Rongli Liu}

\address{Rongli Liu -- Department of Mathematics, Nanjing University, Nanjing, 210093, P.R.China}

\email{rongli.liu@gmail.com}

\keywords{Multitype branching forest, coding, random walks, ballot theorem, total population, cyclic exchangeablity.}

\subjclass[2010]{60C05, 05C05}

\thanks{This work was supported by MODEMAVE research project from the R\'egion
Pays de la Loire.}
\thanks{Ce travail a b\'en\'eci\'e d'une aide de l'Agence Nationale de la Recherche portant la
r\'ef\'erence ANR-09-BLAN-0084-01.}

\date{\today}

\begin{document}
\begin{picture}(0,0)
\put(35,60){\makebox[145truemm][r]{\footnotesize\it To appear in Transactions
 of the American Mathematical Society}}
\end{picture}


\begin{abstract} By extending the breadth first search algorithm to any $d$-type critical or subcritical irreducible branching forest, we show
that such forests can be encoded through $d$ independent, integer valued, $d$-dimensional random walks. An application of this coding together
with a multivariate extension of the Ballot Theorem which is obtained here, allow us to give an explicit form of the law of the total population,
jointly with the number of subtrees of each type, in terms of the offspring distribution of the branching process.
\end{abstract}

\maketitle

\vspace*{.3in}

\section{Introduction}\label{intro}

Let $u_1,u_2\dots$ be the labeling in the breadth first search order of the vertices of a critical or subcritical branching forest with progeny
distribution $\nu$. Call $p(u_i)$, the size of the progeny of the $i$-th vertex, then the stochastic process $(X_n)_{n\ge0}$ defined by,
\[X_{0}=0\;\;\;\mbox{and}\;\;\;X_{n+1}-X_n=p(u_{n+1})-1\,,\;\;\;n\ge0\]
 is a downward skip free random walk with step distribution $P(X_1=n)=\nu(n+1)$, from which the entire structure of the
original branching forest can be recovered. We will refer to this random walk as the
{\it Lukasiewicz-Harris coding path} of the branching forest, see Section 6 of \cite{ha2}, Section 1.1 of \cite{le} or Section 6.2 of \cite{pi}.
A nice example of application of this coding is that the total population of the first $k$ trees ${\bf t}_1,{\bf t}_2,\dots,{\bf t}_k$ of the forest,
see Figure \ref{forest}, may be expressed as the first passage time $T_k$ of $(X_n)_{n\ge0}$ at level $-k$, that is,
\[T_k=\inf\{i:X_i=-k\}\,.\]
This result combined with the following Kemperman's identity (also known as the Ballot Theorem, see \cite{ta}, Lemma 5 in \cite{be} or
Section 6.2 in \cite{pi}):
\[P(T_k=n)=\frac knP(X_n=-k)\,,\]
allows us to compute the law of the total population of ${\bf t}_1,{\bf t}_2,\dots,{\bf t}_k$ in terms of the progeny distribution $\nu$.
Note that the total population is actually a functional of  the associated branching process, $(Z_n,\,n\ge0)$, since the random variable $Z_n$ represents
the number of individuals at the $n$-th generation in the forest. The expression of this law  was first obtained by Otter \cite{ot} and
Dwass \cite{dw}.
\begin{thm}[Otter (49) and Dwass (69)]\label{dwass} Let $Z=(Z_n)$ be a critical or subcritical branching process. Let $\p_k$ be its law when
it starts from $Z_0=k\ge1$ and denote by $\nu$ its progeny law. Let $O$ be the total size of the population generated by $Z$, that is
$O=\sum_{n\ge0}Z_n$. Then for any $n\geq k$,
  \begin{equation}\label{one type prob}
  \p_k(O=n)=\frac{k}{n}\nu^{*n}(n-k)\,,
  \end{equation}
where $\nu^{*n}$ is the $n$-th iteration of the convolution product of the probability $\nu$ by itself.
\end{thm}


\begin{figure}[hbtp]

\vspace*{-5cm}


\hspace*{-1cm}
\includegraphics[height=380pt,width=540pt]{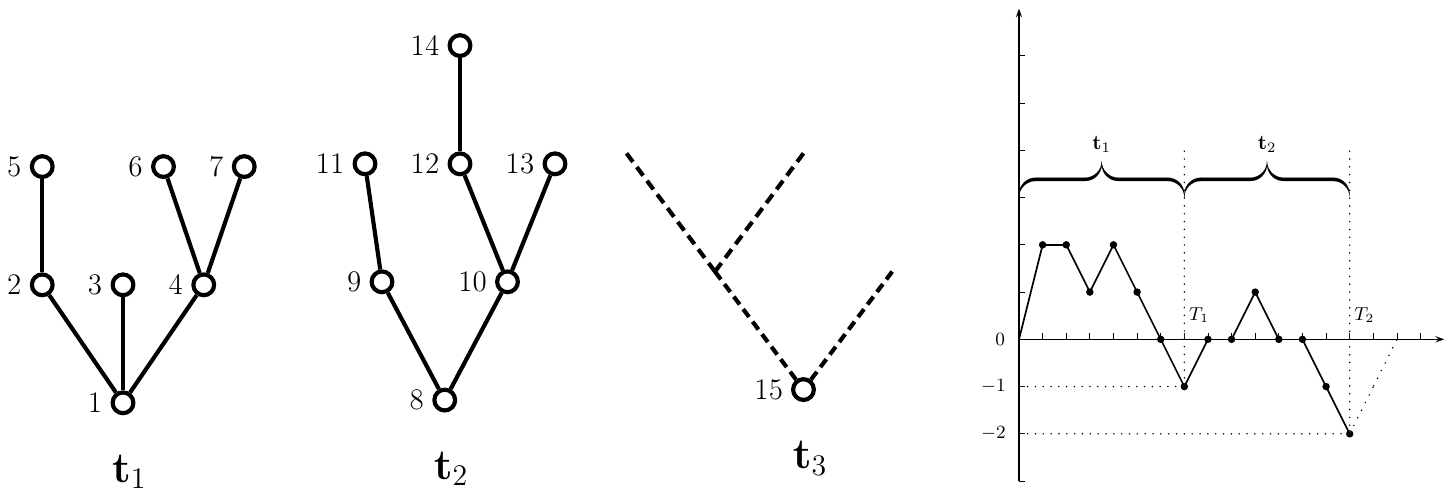}

\vspace{-2.4cm}

\caption{A forest labeled according to the breath first search order and the associated Lukasiewicz-Harris coding path.}\label{forest}
\end{figure}

\noindent More generally, whenever a functional of the branching forest admits a 'nice' expression in terms of the Lukasiewicz-Harris coding
path, we may expect to obtain an explicit form of its law. For instance, the law of the number of individuals with a given degree in the first
$k$ trees can be obtained in this way. We refer to Proposition 1.6 in \cite{ik} where the law of the number of leaves, first obtained in
\cite{min}, is derived from the Lukasiewicz-Harris coding.\\

The goal of this paper is to extend the above program to the multitype case.  The Lukasiewicz-Harris coding will first be extended to multitype forests and will lead to the bijection stated in Theorem \ref{coding} between forests and some set of coding sequences. Then in order to obtain the multitype Otter-Dwass identity which
is stated in Theorem \ref{main}, we first need the equivalent of the Ballot Theorem. This theorem together with its equivalent deterministic
form, the multivariate Cyclic Lemma, are actually amongst the most important results of this paper. Both results require more preliminary notation and will be stated further in the text, see Lemma \ref{crucial} and Theorem \ref{ballot}.\\

Let us first set some definitions and notation in multitype branching processes. We set $\mathbb{Z}_+=\{0,1,2,\dots\}$ and $\mathbb{N}=\{1,2,\dots\}$, and for
any integer $n\ge1$, the set $\{1,\dots,n\}$ will be denoted by $[n]$. In all the sequel of this paper, $d$ will be an integer such that $d\ge 2$. On a probability space
$(\Omega,\mathcal{G},P)$, we define a $d$-type branching process ${\bf Z}:=\{(Z^{(1)}_n,\dots,Z^{(d)}_n),\,n\geq 0\}$, as a $\mathbb{Z}_+^d$ valued Markov
chain with transition probabilities:
\[P({\bf Z}_{n+1}=(k_1,\dots,k_d)\,|\,{\bf Z}_n=(r_1,\dots,r_d))=\nu_1^{*r_1}*\dots*\nu_d^{*r_d}(k_1,\dots,k_d)\,,\]
where $\nu_i$ are distributions on $\mathbb{Z}_+^d$ and $\nu_i^{*r}$ is the $r$-th iteration of the convolution product of $\nu_i$ by itself, with
$\nu_i^{*0}=\delta_0$. For ${\rm r}=(r_1,\dots,r_d)\in\mathbb{Z}_+^d$, we will denote by $\p_{{\rm r}}$
the probability law $P(\,\cdot\,|\,{\bf Z}_0={\rm r})$. The vector $\nu=(\nu_1,\dots,\nu_d)$ will be called the progeny distribution of ${\bf Z}$.
According to this process, each individual of type $i$ gives birth to a random number of children with law $\nu_i$, independently of the other
individuals of its generation. The integer valued random variable $Z_n^{(i)}$ is the total number of individuals of type $i$, at generation $n$.
For $i,j\in [d]$, let us define the rate
\[m_{ij}=\sum_{{\bf z}\in\mathbb Z_+^d}z_j\nu_i({\bf z})\,,\]
that corresponds to the mean number of children of type $j$, given by an individual of type $i$ and let
\[{\bf M}:=(m_{ij})_{i,j\in [d]}\]
be the mean matrix of ${\bf Z}$. Suppose that the extinction time $T$ is a.s.~finite, that is
\begin{equation}\label{3471}
T=\inf\{n:{\bf Z}_n=0\}<\infty\,,\;\;\;\mbox{a.s.}
\end{equation}
Then let $O_i$ be the total number of individuals of type $i$ which are born up to time $T$ (including individuals of the first generation):
\[O_i=\sum_{n=0}^TZ_n^{(i)}=\sum_{n\ge0}Z_n^{(i)}\,.\]
The vector $(O_1,\dots,O_d)$ will be called the {\it total population} of the multitype branching process.\\

Up to now,  most of the results on the exact law of the total population of multitype branching processes concern non irreducible, 2-type branching processes. Let us now
recall them. In the case where $d=2$ and when $m_{12}>0$ and $0<m_{11}\le 1$ but $m_{22}=m_{21}=0$, it may be derived from Theorem~1~$(ii)$ in \cite{be},
that the distribution of the total population of ${\bf Z}$ is given by
\begin{equation}\label{2583}
\p_{(r_1,0)}(O_1=n_1, O_2=n_2)=\frac{r_1}{n_1}\nu_1^{*n_1}(n_1-r_1, n_2)\,,\;\;\;1\leq r_1\leq n_1\,.
\end{equation}
When $m_{12}>0$ and $0< m_{11}, m_{22}\le 1$ but $m_{21}=0$, after some elementary computation, combining the identities in
\eqref{one type prob} and \eqref{2583}, we  obtain that for $n_2\ge1$,
\begin{equation}\label{5215}
\p_{(r_1,0)}(O_1=n_1, O_2=n_2)=\frac{r_1}{n_1n_2}\sum_{j=0}^{n_2}j\nu_1^{*n_1}(n_1-r_1, j)\nu_2^{*n_2}(0, n_2-j).
\end{equation}
Note that (\ref{2583}) and (\ref{5215}) concern only the reducible case, when $d=2$ and $T<\infty$, a.s. As far as we know, those are the
only situations where the law of the total population of multitype branching processes is known explicitly.\\

Recall that if ${\bf M}$ is irreducible, then according to Perron-Frobenius Theorem, it admits a unique
eigenvalue $\rho$ which is simple, positive and with maximal modulus. In this case, we will also say that ${\bf Z}$ is irreducible.
If moreover, ${\bf Z}$ is non-degenarate, that is, if individuals have exactly one offspring with probability different from 1, then
extinction, that is (\ref{3471}), holds
if and only if $\rho\le1$, see \cite{ha1}, \cite{mo} and Chapter V of \cite{an}. If $\rho=1$, we say that ${\bf Z}$ is critical and
if $\rho<1$, we say that ${\bf Z}$ is subcritical. The results of this paper will be concerned by the case where ${\bf Z}$ is
{\it irreducible, non-degenarate, and critical or subcritical} so that (\ref{3471}) holds, that is the multitype branching process
${\bf Z}$ becomes extinct with probability 1. However, let us emphasize that this assumption is only made for simplicity reasons
and that all the proofs can be adapted to the case where the process is supercritical and/or reducible.\\

The next result gives the joint law of the total population together with the total number of individuals of type $j$, whose parent
is of type $i$, $i\neq j$, up to time $T$. Let us denote by $A_{ij}$ this random variable. We emphasize that the variables $A_{ij}$ are
not functionals of the multitype branching process ${\bf Z}$. So, their formal definition and the computation of their law require a more
complete information provided by the forest. Theorem \ref{dwass} and identity \eqref{5215} are extended as follows:

\begin{thm}\label{main}
Assume that the $d$-type branching process ${\bf Z}$ is irreducible, non-degenarate and critical or subcritical. For
$i,j\in[d]$, let $O_i$ be the total number of individuals of type $i$, up to the extinction time $T$ and for $i\neq j$, let $A_{ij}$ be
the total number of individuals of type $j$, whose parent is of type $i$, up to time $T$.

Then for all integers $r_i$, $n_i$, $k_{ij}$, $i,j\in [d]$, such that $r_i\ge0$, $r_1+\dots+r_d\ge1$, $k_{ij}\ge0$,
for $i\neq j$, $-k_{jj}=r_j+\sum_{i\neq j}k_{ij}$, and $n_i\ge -k_{ii}$,
\begin{eqnarray*}
&&\p_{{\rm r}}\Big(O_1=n_1,\dots, O_d=n_d, A_{ij}=k_{ij}, i,j\in [d],i\neq j\Big)\\
&=&\frac{\mbox{\rm det}(K)}{\bar{n}_1\bar{n}_2\dots \bar{n}_d}\prod_{i=1}^d\nu_i^{*n_i}(k_{i1},\dots,k_{i(i-1)},n_i+k_{ii},k_{i(i+1)},\dots,k_{id})\,,
\end{eqnarray*}
where ${\rm r}=(r_1,\dots,r_d)$, $\nu_i^{*0}=\delta_0$, $\bar{n}_i=n_i\vee 1$ and $K$ is the matrix $(-k_{ij})_{i,j\in [d]}$ to which we removed the
line $i$ and the column $i$, for all $i$ such that $n_i=0$.
\end{thm}
\noindent Our proof of Theorem \ref{main} uses a bijection, displayed in Theorem \ref{coding}, between multitype forests and a particular set
of multidimensional, integer valued sequences. A consequence of this result is that any critical or subcritical irreducible multitype branching
forest is  encoded by $d$ independent, $d$-dimensional random walks, see Theorem \ref{Lukasiewicz-Harris}. Then, in a similar way to the single
type case, the total population, jointly with the number of subtrees of  each type in the forest, is expressed as the first passage time of this
multivariate process in some domain. The extension of the Ballot Theorem obtained in Theorem \ref{ballot} allows us to conclude as in the single
type case. Another analogy with the single type case is that the multivariate Lagrange inversion formula known as the Lagrange-Good formula, see
\cite{go}, can be derived from Theorem \ref{main} by applying this theorem to the generating function of the random vector
$(O_1,\dots,O_d)$.\\

This paper is organized as follows. Section \ref{deterministic} is devoted to deterministic multitype forests. In Subsection \ref{space}, we
present the space of these forests and in Subsection \ref{codingforests}, we define the space of the coding sequences and we obtain the
bijection between this space and the space of multitype forests. This result is stated in Theorem \ref{coding}.
Then in Section \ref{random}, we define the probability space of multitype branching forests, we display their multitype
Lukasiewicz-Harris coding in Theorem \ref{Lukasiewicz-Harris} and we prove its application to the total population that is stated
in Theorem \ref{main}. This result requires a multivariate extension of the Ballot Theorem, see Theorem \ref{ballot}, whose proof bears on
the crucial combinatorial Lemma \ref{crucial}. The latter is proved in Section \ref{annex}.

\section{Multitype forests}\label{deterministic}

\subsection{The space of multitype forests}\label{space}

A {\it plane forest}, is a directed planar graph with no loops ${\bf f}\subset {\bf v}\times {\bf v}$, with a finite or infinite set of vertices
${\bf v} = {\bf v} ({\bf f})$, such that the outer degree of each vertex is equal to 0 or 1 and whose connected components, which are called
the {\it trees}, are finite. A forest consisting of a single connected component is also called a tree. In a tree ${\bf t}$,
the only vertex with outer degree equal to 0 is called the {\it root} of ${\bf t}$. It will be denoted by $r({\bf t})$. The roots of the
connected components of a forest ${\bf f}$ are called the roots of ${\bf f}$. For two vertices $u$ and $v$ of a forest ${\bf f}$,
if $(u,v)$ is a directed edge of ${\bf f}$, then we say that $u$ is a {\it child} of $v$, or that $v$  is the {\it parent} of $u$. The set of plane
forests will be denoted by $\mathcal{F}$. The elements of $\mathcal{F}$ will simply be called forests.\\

We will sometimes have to label the forests, which will be done in the following way. We first give an order to the trees of the forest
${\bf f}$ and denote them by ${\bf t}_1({\bf f}),{\bf t}_2({\bf f}),\dots,{\bf t}_k({\bf f}),\dots$ (we will usually write
${\bf t}_1,{\bf t}_2,\dots,{\bf t}_k,\dots$ if no confusion is possible). Then each tree is labeled according to the {\it breadth first
search algorithm}: we read the tree from its root to its last generation by running along each generation from the left to the right.
This definition should be obvious from the example of Figure \ref{forest}. If a forest ${\bf f}$ contains at least $i$ vertices, then the
$i$-th vertex of ${\bf f}$ is denoted by $u_i({\bf f})$. When no confusion is possible, we will simply denote the $i$-th vertex by $u_i$.\\

Recall that $d$ is an integer such that $d\ge 2$. To each forest ${\bf f}\in\mathcal{F}$, we associate an application
$c_{\bf f}:{\bf v}({\bf f})\rightarrow [d]$ such that in the labeling defined above, if $u_i,u_{i+1},\dots,u_{i+j}\in{\bf v}({\bf f})$ have
the same parent, then $c_{\bf f}(u_i)\le c_{\bf f}(u_{i+1})\le\dots\le c_{\bf f}(u_{i+j})$.
For $v\in{\bf v}({\bf f})$, the integer $c_{\bf f}(v)$ is called the {\it type} (or the {\it color}) of $v$. The couple $({\bf f},c_{\bf f})$
is called a {\it $d$-type forest}. When no confusion is possible, we will simply write ${\bf f}$. The set of $d$-type
forests will be denoted by $\mathcal{F}_d$. We emphasize that although there is an underlying labeling for each forest, $\mathcal{F}$ and
$\mathcal{F}_d$ are sets of {\it unlabeled} forests. A 2-type forest is represented on Figure \ref{twotypes} below.\\

\begin{figure}[hbtp]

\vspace*{-3.5cm}

\hspace*{-0.5cm}{\includegraphics[height=350pt,width=500pt]{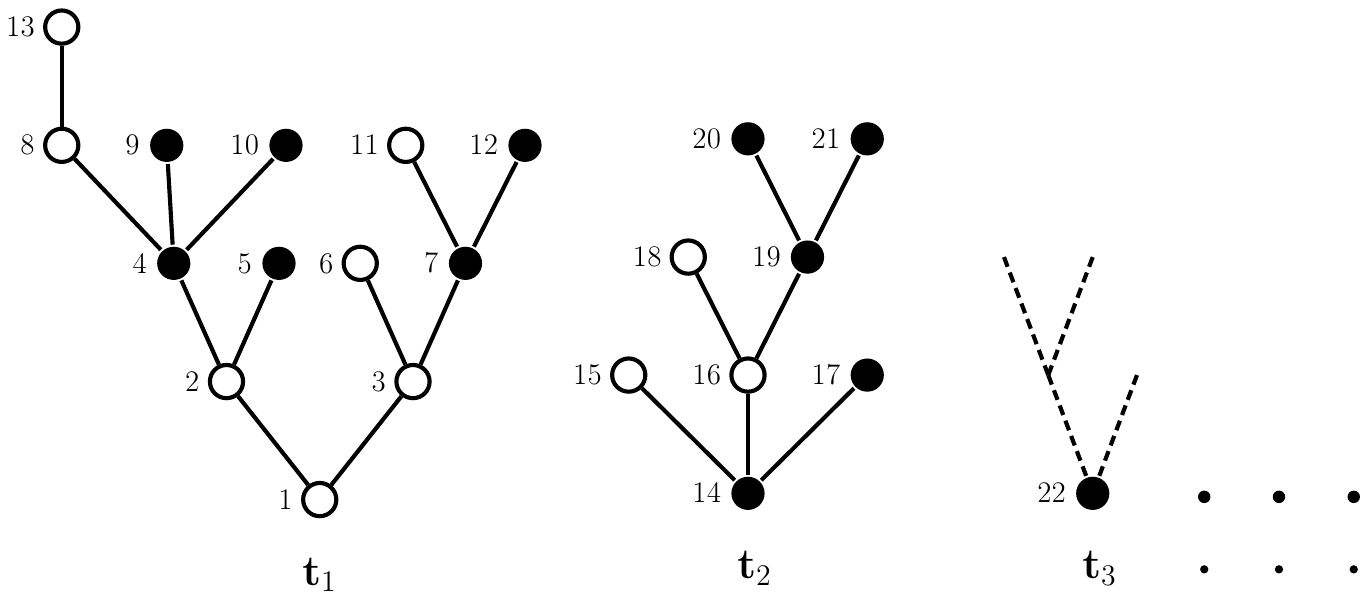}}

\vspace{-3cm}

\caption{A two type forest labeled according to the breath first search order. Vertices of type 1 (resp.~2) are represented in white (resp.~black).}\label{twotypes}
\end{figure}

A {\it subtree of type $i\in[d]$} of a $d$-type forest $({\bf f},c_{\bf f})\in\mathcal{F}_d$ is a maximal connected subgraph of $({\bf f},c_{\bf f})$
whose all vertices are of type $i$. Formally, ${\bf t}$ is a subtree of type $i$ of $({\bf f},c_{\bf f})$, if it is a connected subgraph whose all vertices
are of type $i$ and such that either $r({\bf t})$ has no parent or the type of its parent is different from $i$. Moreover, if the parent of a
vertex $v\in {\bf v}({\bf t})^c$ belongs to ${\bf v}({\bf t})$, then $c_{\bf f}(v)\neq i$. Subtrees of type $i$ of $({\bf f},c_{\bf f})$ are ranked according to the
order of their roots in ${\bf f}$ and are denoted by ${\bf t}^{(i)}_1,{\bf t}^{(i)}_2,\dots,{\bf t}^{(i)}_k,\dots$. The forest
${\bf f}^{(i)}:=\{{\bf t}^{(i)}_1,{\bf t}^{(i)}_2,\dots,{\bf t}^{(i)}_k,\dots\}$ is called {\it the subforest of type $i$} of $({\bf f},c_{\bf f})$. It may be considered
as an element of $\mathcal{F}$. We denote by $u_1^{(i)},u_2^{(i)},\dots$ the elements of ${\bf v}({\bf f}^{(i)})$, ranked
in the breadth first search order of ${\bf f}^{(i)}$. The subforests of type 1 and 2 of a 2-type forest are represented in Figure~\ref{subforests}.\\
\begin{figure}[hbtp]

\vspace*{-2.5cm}

\hspace*{-1cm}{\includegraphics[height=350pt,width=500pt]{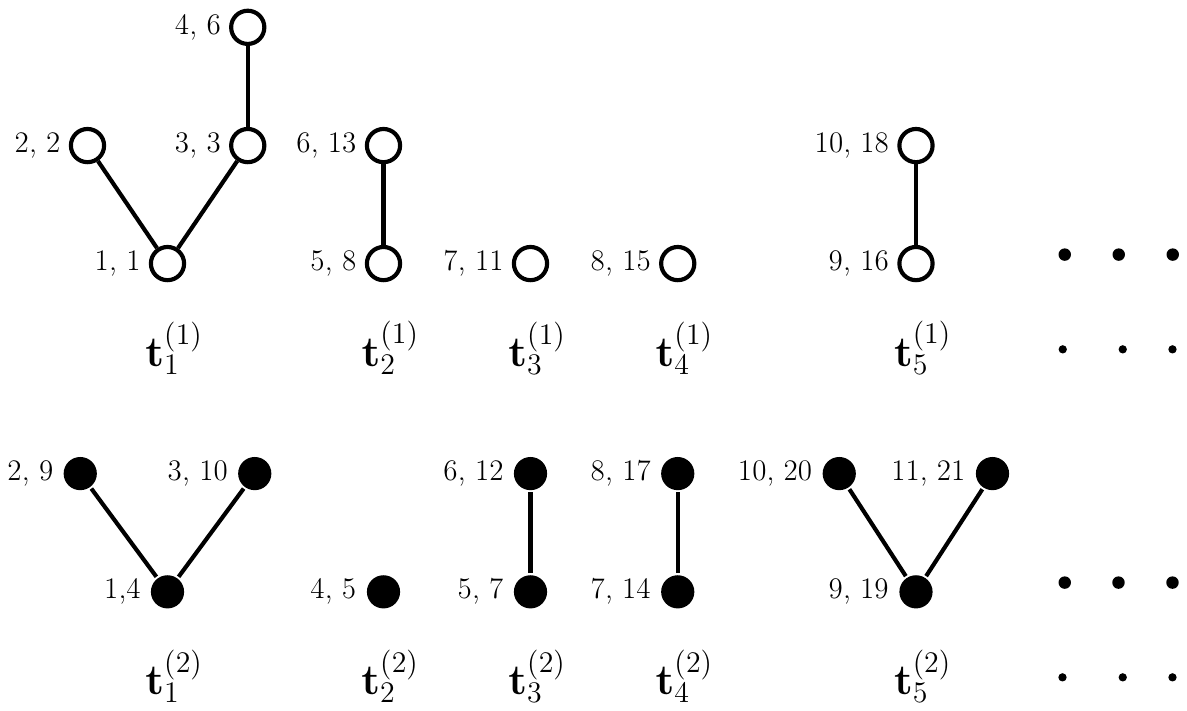}}

\vspace{-2.5cm}

\caption{Labeled subforests ${\bf f}^{(1)}$ and ${\bf f}^{(2)}$ associated with Figure~\ref{twotypes}. Beside each vertex, the first number corresponds to
its rank in ${\bf f}^{(i)}$, $i=1,2$ and the second one is its rank in the original forest.}\label{subforests}
\end{figure}

To any forest $({\bf f},c_{\bf f})\in\mathcal{F}_d$, we associate the {\it reduced forest}, denoted by $({\bf f_r},c_{\bf f_r})\in\mathcal{F}_d$,
which is the forest of $\mathcal{F}_d$ obtained by aggregating all the vertices of each subtree of $({\bf f},c_{\bf f})$ with a given type, in a single
vertex with the same type, and preserving an edge between each pair of connected subtrees.
An example is given in Figure \ref{roots}.

\begin{figure}[hbtp]

\vspace*{-5cm}

\hspace*{-1cm}{\includegraphics[height=350pt,width=500pt]{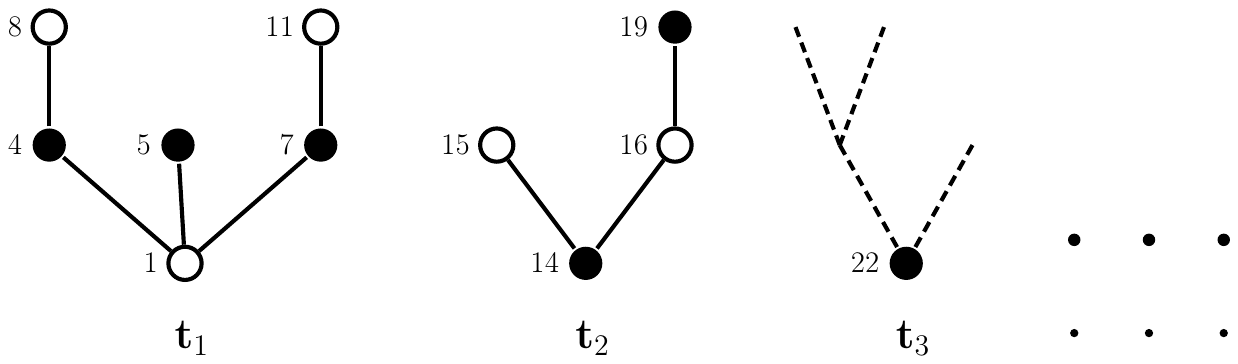}}

\vspace{-3cm}

\caption{Reduced forest associated with the example of Figure \ref{twotypes}. Beside each vertex is the rank of the root of the corresponding
subtree in the original forest.}\label{roots}
\end{figure}

\subsection{Coding multitype forests}\label{codingforests}
For a forest $({\bf f},c_{\bf f})\in\mathcal{F}_d$ and $u\in{\bf v}({\bf f})$, when no confusion is possible,
we denote by $p_i(u)$ the number of children of type $i$ of $u$. For each $i\in [d]$, let $n_i\ge0$ be the number of vertices in the subforest
${\bf f}^{(i)}$ of $({\bf f},c_{\bf f})$. Then let us define the $d$-dimensional chain $x^{(i)}=(x^{i,1},\dots,x^{i,d})$, with length $n_i$ and
whose values belong to the set $\mathbb{Z}^{d}$, by $x_0^{(i)}=0$ and if $n_i\ge1$,
\begin{equation}\label{6245}
x_{n+1}^{i,j}-x_n^{i,j}=p_j(u_{n+1}^{(i)})\,,\;\;\mbox{if $i\neq j\;$ and}\;\;\;x_{n+1}^{i,i}-x_n^{i,i}=
p_i(u_{n+1}^{(i)})-1\,,\;\;\;0\le n\le n_i-1\,,
\end{equation}
where we recall that $(u^{(i)}_n)_{n\ge1}$ is the labeling of the subforest ${\bf f}^{(i)}$
in its own breadth first search order.
Note that the chains $(x_n^{i,j})$, for $i\neq j$ are nondecreasing whereas $(x_n^{i,i})$ is a downward skip free chain, i.e.
$x_{n+1}^{i,i}-x_n^{i,i}\ge -1$, for $0\le n\le n_i-1$. The chain $(x_n^{i,i})$ corresponds to the Lukasiewicz-Harris coding walk of
the subforest ${\bf f}^{(i)}$, as defined in the introduction, see also Section 6.2 in \cite{pi} for a proper definition.
In particular, if $n_i$ is finite, then $n_i=\min\{n:x^{i,i}_{n}=\min_{0\le k\le n_i} x^{i,i}_k\}$. These properties of the chains
$x^{(i)}$ lead us to the following definition.
\begin{definition}\label{9076}
Let $S_d$ be the set of $\left[\mathbb{Z}^{d}\right]^d$-valued sequences, $x=(x^{(1)},x^{(2)},\dots,x^{(d)})$, such that
for all $i\in [d]$, $x^{(i)}=(x^{i,1},\dots,x^{i,d})$ is a $\mathbb{Z}^{d}$-valued sequence defined on some
interval of integers, $\{0,1,2,\ldots, n_i\}$, $0\le n_i\le\infty$, which satisfies $x^{(i)}_0=0$ and if $n_i\ge1$ then
\begin{itemize}
\item[$(i)$]  for $i\neq j$, the sequence $(x_n^{i,j})_{0\le n\le n_i}$ is nondecreasing,
\item[$(ii)$] for all $i$, $x_{n+1}^{i,i}-x_n^{i,i}\ge -1$, $0\le n\le n_i-1$.
\end{itemize}
A sequence $x\in S_d$ will sometimes be denoted by $x=(x^{i,j}_k,\,0\le k\le n_i,\,i,j\in [d])$ and for more convenience, we will sometimes denote $x_k^{i,j}$ by $x^{i,j}(k)$. The vector
${\rm n}=(n_1,\dots,n_d)\in\overline{\mathbb{Z}}_+^d$, where $\overline{\mathbb{Z}}_+=\mathbb{Z}_+\cup\{+\infty\}$ will be called the length of $x$.
\end{definition}
\noindent Relation (\ref{6245}) defines an application from the set $\mathcal{F}_d$ to the set $S_d$. Let us denote by
$\Psi$ this application, that is
\begin{eqnarray}
\Psi:\mathcal{F}_d&\rightarrow&S_d\label{7809}\\
({\bf f},c_{\bf f})&\mapsto&\Psi(({\bf f},c_{\bf f}))=x\,.\nonumber
\end{eqnarray}
For $x\in S_d$, set $k_i=-\inf_{0\le n\le n_i}x^{i,i}_n$ and define the {\it first passage time process} of the chain $(x_n^{i,i})$ as
follows:
\begin{equation}\label{5170}
\tau^{(i)}_k=\min\{n\ge0: x^{i,i}_n=-k\}\,,\;\;0\le k\le k_i\,,
\end{equation}
where $\tau^{(i)}_{k_i}=\infty$, if $k_i=\infty$. If $x$ is the image by $\Psi$ of a forest  $({\bf f},c_{\bf f})\in\mathcal{F}_d$, i.e.
$x=\Psi(({\bf f},c_{\bf f}))$, then $k_i$ is the (finite or infinite) number of trees in the subforest ${\bf f}^{(i)}$ and for $k<\infty$,
the time $\tau^{(i)}_k$ is the total number of vertices which are contained in the first $k$ trees of ${\bf f}^{(i)}$, i.e.
${\bf t}^{(i)}_1,{\bf t}^{(i)}_2,\dots,{\bf t}^{(i)}_k$. This fact is well known and easily follows from the Lukasiewicz-Harris coding
of the single type forest ${\bf f}^{(i)}$, see the introduction and Lemma 6.3 in \cite{pi}. Then for $i,j\in [d]$, define the
integer valued sequence
\begin{equation}\label{5169}\bar{x}^{i,j}_k=x^{i,j}(\tau^{(i)}_k)\,,\;\;\;0\le k\le k_i\,.
\end{equation}
If $x=\Psi(({\bf f},c_{\bf f}))$, then we may check that when $i\neq j$, $\bar{x}^{i,j}_k$ is the number of subtrees of type $j$
whose root is the child of a vertex
in ${\bf t}^{(i)}_1,{\bf t}^{(i)}_2,\dots,{\bf t}^{(i)}_k$. Or equivalently, it is the number of vertices of type $j$ whose
parent is a vertex of ${\bf t}^{(i)}_1,{\bf t}^{(i)}_2,\dots,{\bf t}^{(i)}_k$. For each $i\in [d]$, we set
\[\bar{x}^{(i)}=(\bar{x}^{i,1},\dots,\bar{x}^{i,d})\;\;\;\mbox{and}\;\;\;
\bar{x}=(\bar{x}^{(1)},\bar{x}^{(2)},\dots,\bar{x}^{(d)})\,.\]
Clearly for $i\neq j$, the sequence $(\bar{x}^{i,j}_k)_{0\le k\le k_i}$ is nondecreasing and $\bar{x}^{i,i}_k=-k$, for all $i\in [d]$
and $0\le k\le k_i$. Therefore $\bar{x}\in S_d$ and recalling the definition of the reduced forest, $({\bf f_r},c_{\bf f_r})$, see the end of
 Section \ref{space}, we may check that:
\begin{equation}\label{7345}
\Psi(({\bf f_r},c_{\bf f_r}))=\bar{x}\,.
\end{equation}
For a forest $({\bf f},c_{\bf f})\in\mathcal{F}_d$ with trees ${\bf t}_1,{\bf t}_2,\dots$, we will denote by
${\rm c}_{({\bf f},c_{\bf f})}$ the sequence of types of the roots of ${\bf t}_1,{\bf t}_2,\dots$, i.e.
\[{\rm c}_{({\bf f},c_{\bf f})}:=(c_{{\bf f}}(r({\bf t}_1)),c_{{\bf f}}(r({\bf t}_2)),\dots)\,.\]
Note that ${\rm c}_{({\bf f},c_{\bf f})}\in\cup_{1\le r\le\infty}[d]^r$ and that
${\rm c}_{({\bf f},c_{\bf f})}={\rm c}_{({\bf f_r},c_{\bf f_r})}$.
When no confusion is possible, ${\rm c}_{({\bf f},c_{\bf f})}$ will simply be denoted by ${\rm c}=(c_1,c_2,\dots)$
and we will call it the {\it root type sequence} of the forest.\\

Then before we state the general result on the coding of multitype forests in Theorem \ref{coding}, we first need to
show that the sequences $(\bar{x}^{i,j})_{i\neq j}$ together with ${\bf\rm c}=(c_1,c_2,\dots)$ allow us to
encode the reduced forest $({\bf f_r},c_{\bf f_r})$, i.e. this forest can be reconstructed from $(\bar{x},{\rm c})$. This claim is stated
in Lemma \ref{numbersubtrees} below. In order to prove it, we first need to describe the set of sequences which
encode reduced forests and to state the preliminary Lemma \ref{preliminary} regarding these sequences.\\

Recall that $\overline{\mathbb{Z}}_+=\mathbb{Z}_+\cup\{+\infty\}$ and let us  define the following (non total) order in
$\overline{\mathbb{Z}}_+^d$: for two elements ${\rm q}=(q_1,\dots,q_d)$ and
${\rm q}'=(q_1',\dots,q_d')$ of $\overline{\mathbb{Z}}_+^d$ we write ${\rm q}\le {\rm q'}$ if $q_i\le q'_i$ for all $i\in [d]$.
Moreover we write ${\rm q}< {\rm q'}$ if ${\rm q}\le {\rm q'}$ and if there is $i\in [d]$ such that $q_i<q'_i$.
For an element $x=(x^{i,j}_k,\,0\le k\le q_i,\,i,j\in [d])$ of $S_d$ with length ${\rm q}=(q_1,\dots,q_d)$, and
${\rm r}=(r_1,\dots,r_d)\in\mathbb{Z}_+^d$, we say that the system of equations $({\rm r},x)$ admits a solution if
there exists ${\rm s}=(s_1,\dots,s_d)\in\mathbb{Z}_+^d$, such that ${\rm s}\le {\rm q}$ and
\begin{equation}\label{2418}
r_j+ \sum_{i=1}^d x^{i,j}(s_i)=0\,,\;\;j=1,2,\ldots,d\,.
\end{equation}
We will see in Lemma \ref{numbersubtrees} and Theorem \ref{coding} that for any finite forest $({\bf f},c_{\bf f})\in \mathcal{F}_d$ with $r_i$ roots of type $i$, the length ${\rm q}$ of $x=\Psi(({\bf f},c_{\bf f}))$ is a solution of $({\rm r},x)$ and this solution is the smallest one in a sense that is specified in the following lemma.

\begin{lemma}\label{preliminary} Let $x\in S_d$ and ${\rm r}=(r_1,\dots,r_d)\in\mathbb{Z}_+^d$.
Assume that the system $({\rm r},x)$ admits a solution, then
\begin{itemize}
\item[$(i)$] there exists a unique solution ${\rm n}=(n_1,\dots,n_d)$ of the system $({\rm r},x)$
such that if ${\rm n}'$ is any solution of $({\rm r},x)$, then ${\rm n}\le {\rm n}'$. Moreover we have
$n_i=\min\{n:x^{i,i}_n=\min_{0\le k\le n_i}x^{i,i}_k\}$, for all $i\in [d]$. A solution such as ${\rm n}$
will be called {\em the smallest solution} of the system $({\rm r},x)$.
\item[$(ii)$] Let ${\rm r}'\in\mathbb{Z}_+^d$ be such that ${\rm r}'\le{\rm r}$. Then the system $({\rm r}',x)$ admits a solution. Let us denote its
smallest solution by ${\rm n}'$. Then the system $({\rm r}-{\rm r}',\tilde{x})$, where $\tilde{x}^{i,j}(k)={x}^{i,j}(n_i'+k)-{x}^{i,j}(n_i')$,
$0\le k\le n_i-n_i'$, admits a solution, and its smallest solution is ${\rm n}-{\rm n}'$.
\end{itemize}
\end{lemma}

\noindent A proof of this lemma is given in Section \ref{annex}. For ${\rm r}=(r_1,\dots,r_d)\in\mathbb{Z}_+^d$,
with $r=r_1+\dots+r_d\ge1$, that is ${\rm r}>0$, we define,
\[
C_d^{\rm r}=\Big\{c\in [d]^r:\mbox{Card}\,\{j\in [r]:\ c_j=i\}=r_i,\,i\in [d]\Big\}\,.
\]
We emphasize that the root type sequence of a forest $({\bf f},c_{\bf f})\in\mathcal{F}_d$  with $r=r_1+\dots+r_d$ trees
amongst which exactly $r_i$ trees have a root of type $i$ is an element of $C_d^{\rm r}$. Now we define the subsets of forests and
reduced forests whose root type sequence is in $C_d^{\rm r}$ and that contain at least one vertex of each type.

\begin{definition}\label{4365} Let ${\rm r}=(r_1,\dots,r_d)\in\mathbb{Z}_+^d$, such that ${\rm r}>0$.
\begin{itemize}
\item[$(i)$] We denote by $\mathcal{F}_d^{\rm r}$, the subset of $\mathcal{F}_d$ of forests $({\bf f},c_{\bf f})$ with $r_1+\dots+r_d$ trees,
which contain at least one vertex of each type, and such that ${\rm c}_{({\bf f},c_{\bf f})}\in C_d^{\rm r}$.
\item[$(ii)$]  We denote by $\bar{\mathcal{F}}_d^{\rm r}$ the subset of $\mathcal{F}_d^{\rm r}$, of reduced forests.
More specifically, $({\bf f},c_{\bf f})\in\bar{\mathcal{F}}_d^{\rm r}$ if $({\bf f},c_{\bf f})\in\mathcal{F}_d^{\rm r}$
and if for each $i$, vertices of type $i\in [d]$ in ${\bf v}({\bf f})$ have no child of type $i$.
\end{itemize}
\end{definition}

\noindent Then we define the sets of coding sequences related to $\mathcal{F}_d^{\rm r}$ and $\bar{\mathcal{F}}_d^{\rm r}$.

\begin{definition}\label{4366} Let ${\rm r}=(r_1,\dots,r_d)\in\mathbb{Z}_+^d$ be such that ${\rm r}>0$.
\begin{itemize}
\item[$(i)$] We denote by $S_d^{\rm r}$ the subset of $S_d$ of sequences $x$ whose length belongs to $\mathbb{N}^d$  and corresponds to
the smallest solution of the system $({\rm r},x)$ defined in $(\ref{2418})$.
\item[$(ii)$] We denote by $\bar{S}_d^{\rm r}$ the subset of $S_d^{\rm r}$ consisting in sequences,
such that $x^{i,i}_k=-k$, for all $k$ and $i$.
\end{itemize}
\end{definition}

\noindent Then we first establish a bijection between the sets $\bar{\mathcal{F}}_d^{\rm r}$ and $\bar{S}_d^{\rm r}\times C_d^{\rm r}$.
Recall the definition of $\Psi$ in (\ref{7809}).

\begin{lemma}\label{numbersubtrees} Let  ${\rm r}=(r_1,\dots,r_d)\in\mathbb{Z}_+^d$ be such that ${\rm r}>0$, then
the mapping
\begin{eqnarray*}
\Phi:\bar{\mathcal{F}}_d^{\rm r}&\rightarrow&\bar{S}_d^{\rm r}\times C_d^{\rm r}\\
({\bf f},c_{\bf f})&\mapsto&\left(\Psi(({\bf f},c_{\bf f})),{\rm c}_{({\bf f},c_{\bf f})}\right)
\end{eqnarray*}
is a bijection.
\end{lemma}
\begin{proof} Let $({\bf f},c_{\bf f})\in\bar{\mathcal{F}}_d^{\rm r}$ and let $k_i$ be the total number of subtrees of type $i$ which
are contained in $({\bf f},c_{\bf f})$. (Note that since $({\bf f},c_{\bf f})$ is a reduced forest, its subtrees are actually single vertices.)
By definition, $k_i\ge1$, for each $i$. The fact that, ${\rm c}_{({\bf f},c_{\bf f})}\in C_d^{\rm r}$
follows from Definition \ref{4365} $(ii)$. Then let us show that $x=\Psi(({\bf f},c_{\bf f}))\in\bar{S}_d^{\rm r}$.
Since $({\bf f},c_{\bf f})$ is a reduced forest, then $x=\bar{x}$. Besides, from (\ref{5169}), $x$ has length
${\rm k}=(k_1,\dots,k_d)$ and for $j\neq i$, $x^{i,j}(k_i)$ is the number of subtrees of type $j$ whose root is a child of
a vertex of $\left\{{\bf t}^{(i)}_1,{\bf t}^{(i)}_2,\dots,{\bf t}^{(i)}_{k_i}\right\}$, i.e. of any subtree of type $i$ in
$({\bf f},c_{\bf f})$. Hence for $j\in [d]$, $\sum_{i\neq j}x^{i,j}(k_i)$ is the total number of subtrees of type $j$ in
$({\bf f},c_{\bf f})$, whose root is a child of a vertex of type $i\in [d]$, $i\neq j$.
Then in order to obtain the total number of subtrees of type $j$, it remains to add to $\sum_{i\neq j}x^{i,j}(k_i)$, the number of subtrees of type $j$ whose root is one of the roots of ${\bf t}_1,\dots,{\bf t}_r$, where $r=r_1+\dots+r_d$. The latter number is $r_j$, so that $k_j=r_j+\sum_{i\neq j}x^{i,j}(k_i)$. Since moreover from (\ref{5169}),  $x^{i,i}_{k}=-k$, for all $0\le k \le k_i$, we have proved that
${\rm k}$ is a solution of the system $({\rm r},x)$. It remains to prove that it is the smallest solution.

Let us first assume that ${\rm r}=(1,0,\dots,0)$, so that $({\bf f},c_{\bf f})$ consists in a single tree ${\bf t}_1$ whose root
has color $c_{{\bf f}}(r({\bf t}_1))=1$. Then we can reconstruct, this tree from the $d$ sequences $(x^{(i)}_k,\,0\le k\le k_i)$, $i\in [d]$
by inverting the procedure defined in (\ref{6245}) and this {\it reconstruction procedure} gives a unique tree. Indeed,
by definition of the application $\Psi$, each sequence $(x^{(i)}_k,\,0\le k\le k_i)$, is associated to a unique 'marked subforest',
say $\tilde{{\bf f}}^{(i)}$, of type $i$ whose vertices kept the memory of their progeny. More specifically, for $k\in [k_i]$, the
increment $x^{(i)}_k-x^{(i)}_{k-1}$ gives the progeny of the $k$-th vertex of the subforest ${\bf f}^{(i)}$.
This connection between marked subforests $\tilde{{\bf f}}^{(i)}$ and sequences $(x^{(i)}_k,\,0\le k\le k_i)$
is illustrated on Figure \ref{reconstruction}.

\begin{figure}[hbtp]



\hspace*{-1cm}
\includegraphics[height=430pt,width=280pt]{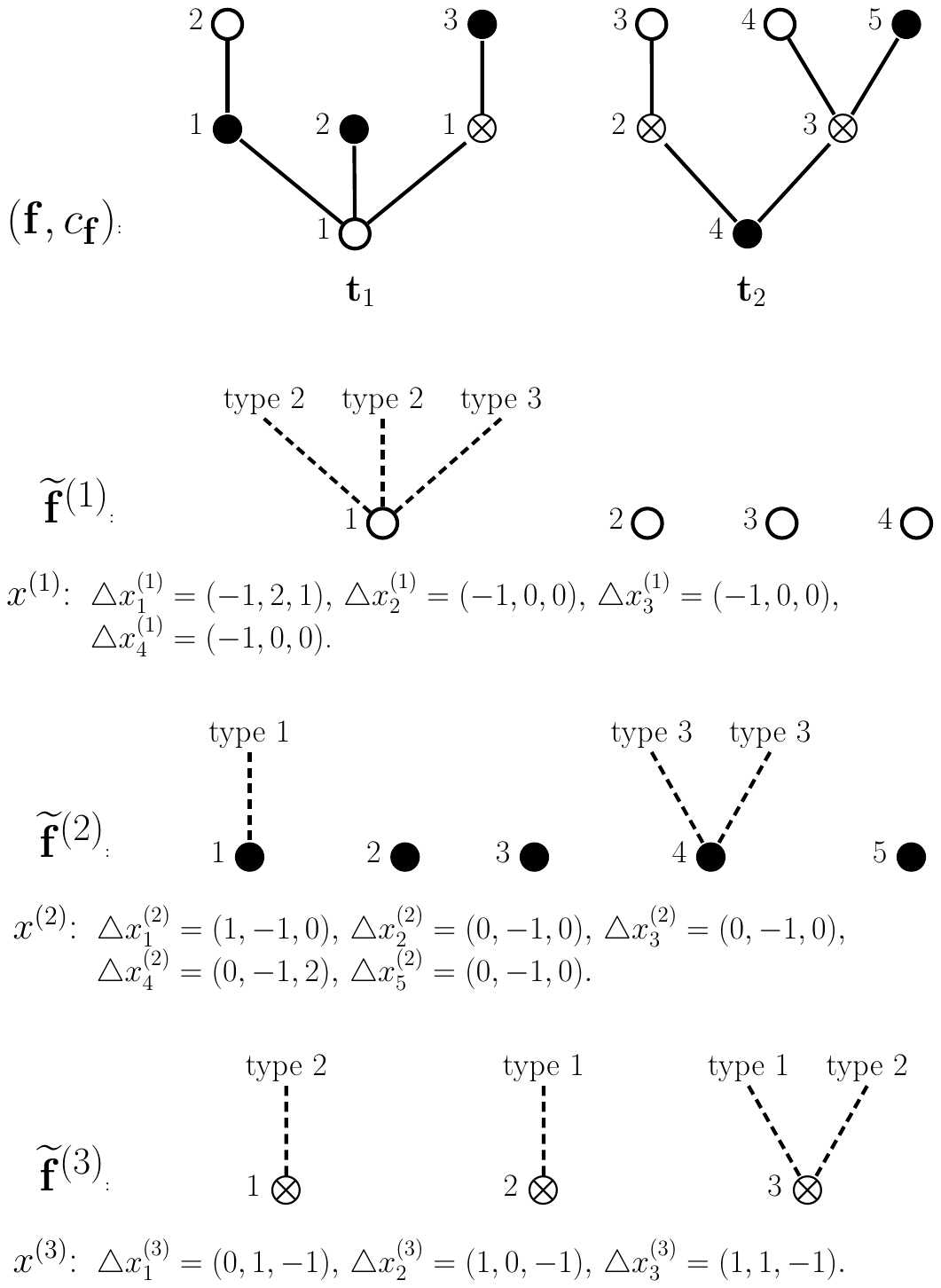}

\vspace{-2cm}

\caption{A three types reduced forest $({\bf f},c_{\bf f})$, the three marked subforests $\tilde{{\bf f}}^{(1)}$, $\tilde{{\bf f}}^{(2)}$, $\tilde{{\bf f}}^{(3)}$ of $({\bf f},c_{\bf f})$ and the coding sequences $x^{(1)}$, $x^{(2)}$, $x^{(3)}$. Here we have set
$\Delta x_k^{(i)}=x_k^{(i)}-x_{k-1}^{(i)}$.}
\label{reconstruction}
\end{figure}

Now let ${\rm k}'\le {\rm k}$ be the smallest solution of the system $({\rm r},x)$.
Let ${\rm q}=(q_1,\dots,q_d)<{\rm k}'$ and suppose that we have been able to perform the reconstruction procedure until ${\rm q}$,
that is from the sequences $(x_k^{(i)},0\le k\le q_i)$, $i\in [d]$. Then since ${\rm q}$ is not a solution of $({\rm r},x)$,
we see from what has been proved just above that the tree that is obtained is 'not complete'. That is, at least one of its leaves (say of type $j$) is marked,
so that this leaf should still get children whose types and numbers are given by the next jump $x_{q_j+1}^{(j)}-x_{q_j}^{(j)}$, for $q_j<k_j'$,
according to the reconstruction procedure. Thus, doing so, we necessarily end up with a tree from the sequences $(x_k^{(i)},0\le k\le k_i')$,
$i\in [d]$, and this tree is complete, that is none of its leaves is marked. Then since the reconstruction procedure obtained by inverting
(\ref{6245}), gives a unique  tree, we necessarily have ${\rm k}'={\rm k}$.

Then let ${\rm r}=(r_1,\dots,r_d)\in\mathbb{Z}_+^d$. Assume with no loss of generality that the root of the
first tree ${\bf t}_1$ of $({\bf f},c_{\bf f})$ has color 1. Let $k_i^1$ be the number of subtrees of type $i$ in  ${\bf t}_1$.
From Lemma \ref{preliminary}, the system $({\rm r}^1,x)$, where ${\rm r}^1:=(1,0,\dots,0)$, admits a smallest solution. Moreover from
the reconstruction procedure which is described above,  this solution is ${\rm k}^1=(k_1^1,\dots,k_d^1)$. Suppose now with no loss of
generality that the second tree,  ${\bf t}_2$ in $({\bf f},c_{\bf f})$ has color 2. Let $k_i^2$ be the number of subtrees of type $i$
in  ${\bf t}_2$. Then from the same arguments as for the reconstruction of the first tree,
${\bf t}_2$ may be reconstructed from the system $({\rm r}^2,y)$, where ${\rm r}^2:=(0,1,0,\dots,0)$
and $y^{i,j}(k)={x}^{i,j}(k_i^1+k)-{x}^{i,j}(k_i^1)$, $k\ge0$. Moreover $({\rm r}^2,y)$ admits
${\rm k}^2=(k_1^2,\dots,k_d^2)$ as a smallest solution. Then from part $(ii)$ of Lemma \ref{preliminary}, ${\rm k}^1+{\rm k}^2$
is the smallest solution of the system $({\rm r}^1+{\rm r}^2,x)$. So we have proved the result for the forest consisting of
the trees ${\bf t}_1$ and ${\bf t}_2$. Then by iterating these arguments for each tree of $({\bf f},c_{\bf f})$, we obtain
that $x\in\bar{S}_d^{\rm r}$.\\

Conversely, let  ${\rm c}=(c_1,\dots,c_r)\in C_d^{\rm r}$, $x\in\bar{S}_d^{\rm r}$ and let ${\rm k}=(k_1,\dots,k_d)$
be the smallest solution of the system $({\rm r},x)$.
Then let us show that there is a forest $({\bf f},c_{\bf f})\in\bar{\mathcal{F}}_d^{\rm r}$ such that
$\Psi(({\bf f},c_{\bf f}))=x$ and ${\rm c}_{({\bf f},c_{\bf f})}={\rm c}$.
Assume, without loss of generality that $c_1=1$. From Lemma \ref{preliminary} $(ii)$, there is a smallest solution, say
${\rm k}^1=(k_1^1,\dots,k_d^1)$, to the  system $({\rm r}^1,x)$, where ${\rm r}^1:=(1,0,\dots,0)$.
Then we may reconstruct a unique forest  $({\bf t}_1,c_{{\bf t}_1})\in\bar{\mathcal{F}}_d^1$ (consisting in a single tree) such that
$\Psi(({\bf t}_1,c_{{\bf t}_1}))=(x^{i,j}_k,0\le k\le k^1_i,\,i,j\in [d])$ and ${\rm c}_{({\bf t}_1,c_{{\bf t}_1})}=1$
by inverting the procedure that is described in (\ref{6245}). Assume for instance that $c_2=2$ and set ${\rm r}^2:=(1,1,0,\dots,0)\le {\rm r}$,
then from Lemma \ref{preliminary} $(ii)$, there is a smallest solution, say ${\rm k}^2$, to the system $({\rm r}^2,x)$.
Moreover, ${\rm k}^3={\rm k}^2-{\rm k}^1$ is the smallest solution of the system $({\rm r}^2-{\rm r}^1,y)$, where
$y^{i,j}(k)={x}^{i,j}(k_i^1+k)-{x}^{i,j}(k_i^1)$, $k\ge0$. Then as before, we can reconstruct a unique tree $({\bf t}_2,c_{{\bf t}_2})$ such
that $\Psi(({\bf t}_2,c_{{\bf t}_2}))=y$ and such that the forest $\hat{{\bf f}}=\{{\bf t}_1,{\bf t}_2\}$ satisfies
$\Psi((\hat{{\bf f}},c_{\hat{{\bf f}}}))=(x^{i,j}_k,0\le k\le k^2_i,\,i,j\in [d])$
and ${\rm c}_{(\hat{{\bf f}},c_{\hat{{\bf f}}})}=(1,1,0,\dots,0)$. Then iterating these arguments, we may reconstruct a unique forest
$({\bf f},c_{\bf f})\in\mathcal{F}_d^{\rm r}$ such that $\Psi(({\bf f},c_{\bf f}))=x$ and ${\rm c}_{({\bf f},c_{\bf f})}=c$.$\;\;\Box$\\
\end{proof}

\noindent Let $x\in S_d$ with length ${\rm n}=(n_1,\dots,n_d)$ and recall from (\ref{5169}), the definition of the
associated sequence $\bar{x}$, with length ${\rm k}=(k_1,\dots,k_d)$, such that $k_i=-\min_{0\le n\le n_i} x_n^{i,i}$.

\begin{lemma}\label{3379} Let ${\rm r}\in\mathbb{Z}^d_+$, such that ${\rm r}>0$ and $x\in S_d$, with length ${\rm n}\in\mathbb{N}^d$
and set $k_i=-\min_{0\le n\le n_i} x_n^{i,i}$, $i\in [d]$. If ${\rm n}$ is the smallest solution of the system $({\rm r},x)$
$($i.e. $x\in S_d^{\rm r})$, then ${\rm k}=(k_1,\dots,k_d)$ is the smallest solution of the system $({\rm r},\bar{x})$. Conversely,
if $n_i=\tau^{(i)}_{k_i}$, for all $i\in [d]$ and if ${\rm k}$ is the smallest solution of $({\rm r},\bar{x})$,
$($i.e. $\bar{x}\in \bar{S}_d^{\rm r})$, then  ${\rm n}$ is the smallest solution of $({\rm r},x)$.
\end{lemma}
\begin{proof} Assume that ${\rm n}$ is the smallest solution of the system $({\rm r},x)$.  Then from part $(i)$ of Lemma \ref{preliminary},
$n_i=\tau^{(i)}_{k_i}$, hence ${\rm k}$ is a solution of $({\rm r},\bar{x})$. Let ${\rm k}'\le{\rm k}$
be such that
\[
r_j+ \sum_{i=1}^d \bar{x}^{i,j}(k_i')=0\,,\;\;j=1,2,\ldots d\,.
\]
Then by definition of $\bar{x}$ there is ${\rm n}'\le {\rm n}$ such that $n_i'=\tau^{(i)}_{k_i'}$ and
\[
r_j+ \sum_{i=1}^d x^{i,j}(n_i')=0\,,\;\;j=1,2,\ldots d\,.
\]
So ${\rm n}'={\rm n}$ and hence  ${\rm k}'={\rm k}$.

The converse is proved in the same way. Suppose that $n_i=\tau^{(i)}_{k_i}$, $i\in [d]$, and that ${\rm k}$ is the smallest solution of
$({\rm r},\bar{x})$. Then clearly, ${\rm n}$ is a solution of $({\rm r},x)$. Let ${\rm n}'$ be the smallest solution of
$({\rm r},x)$. Then  from Lemma \ref{preliminary}, there is ${\rm k}'$ such that $n_i'=\tau^{(i)}_{k_i'}$, hence ${\rm k}'$ is a solution of
$({\rm r},\bar{x})$. This implies that ${\rm k}\le {\rm k}'$, so that ${\rm n}\le {\rm n}'$, hence ${\rm n}={\rm n}'$. $\;\;\Box$\\
\end{proof}

\noindent Now we extend the application $\Phi$ defined in Lemma \ref{numbersubtrees} to the set $\mathcal{F}_d^{\rm r}$.
Here is the main result of this section, that can be considered as an extension of Proposition 1.1 in \cite{le}.

\begin{thm}\label{coding}
Let  ${\rm r}=(r_1,\dots,r_d)\in\mathbb{Z}_+^d$, be such that ${\rm r}>0$, then the mapping
\begin{eqnarray*}
\Phi:\mathcal{F}_d^{\rm r}&\rightarrow&S_d^{\rm r}\times C_d^{\rm r}\\
({\bf f},c_{\bf f})&\mapsto&\left(\Psi(({\bf f},c_{\bf f})),{\rm c}_{({\bf f},c_{\bf f})}\right)
\end{eqnarray*}
is a bijection.
\end{thm}
\begin{proof} Let us first check that for any $({\bf f},c_{\bf f})\in\mathcal{F}_d^{\rm r}$, we have
$\Phi(({\bf f},c_{\bf f}))\in S_d^{\rm r}\times C_d^{\rm r}$. By definition \ref{4365}, $(i)$,
${\rm c}_{({\bf f},c_{\bf f})}\in C_d^{\rm r}$. Now set $x=\Psi(({\bf f},c_{\bf f}))$ and let
$({\bf f_r},c_{\bf f_r})\in\bar{\mathcal{F}}_d^{\rm r}$ be the forest, $({\bf f},c_{\bf f})$ once reduced.
Then from (\ref{7345}) and Lemma \ref{numbersubtrees}, this reduced forest is encoded by $(\bar{x},c_{\bf f})$.
Let  ${\rm k}=(k_1,\dots,k_d)$ be the number of subtrees of type $i$ in this forest (this is actually the number of vertices of type $i$),
then ${\rm k}$ is the length of $\bar{x}$ and it is the smallest solution of $({\rm r},\bar{x})$, i.e. $\bar{x}\in\bar{S}_d^{\rm r}$.
Moreover ${\rm n}=(n_1,\dots,n_d)$, where $n_i=\tau^{(i)}_{k_i}$ is the length of $x$, and from Lemma \ref{3379},
it is the smallest solution of $({\rm r},x)$. So, we have proved that $\Psi(({\bf f},c_{\bf f}))\in S_d^{\rm r}$.

Conversely let $(x,{\rm c})\in S_d^{\rm r}\times C_d^{\rm r}$. From Lemma \ref{numbersubtrees},
to $(\bar{x},{\rm c})$, we may associate a unique forest $({\bf f_r},c_{\bf f_r})\in\bar{\mathcal{F}}_d^{\rm r}$. Then let $k_i$ be the number
of vertices of type $i$ in this forest. For $ k\in [k_i]$, let $u_k^{(i)}$ be the $k$-th vertex of type $i$ in the breadth first search order of
$({\bf f_r},c_{\bf f_r})$. Then in $({\bf f_r},c_{\bf f_r})$, we replace the vertex $u_k^{(i)}$ by the subtree of type $i$ which is encoded
by the Lukasiewicz-Harris path $(x^{i,i}(\tau^{(i)}_{k-1}+l)+k-1\,,\;0\le l\le \tau^{(i)}_k-\tau^{(i)}_{k-1})$. We know about the
progeny of each vertex of this subtree, thanks to the chains $(x^{i,j}(\tau^{(i)}_{k-1}+l)-x^{i,j}(\tau^{(i)}_{k-1})\,,\;0\le l\le \tau^{(i)}_k-\tau^{(i)}_{k-1})$,
so that we can graft at the proper place, on this subtree, all the corresponding subtrees of the other types which have been constructed
from the same procedure. Proceeding this way, we construct a unique forest $({\bf f},c_{\bf f})\in\mathcal{F}_d^{\rm r}$ and we easily check
that $\Psi(({\bf f},c_{\bf f}))=x$.$\;\;\Box$\\
\end{proof}

\section{Multitype branching trees and forests}\label{random}

Let $\nu_i$, $i\in [d]$ be distributions on $\mathbb{Z}_+^d$, such that $\nu=(\nu_1,\dots,\nu_d)$ is the progeny law of
an irreducible, critical or subcritical, non-degenarate branching process, as defined in Section \ref{intro}. Assume that we can define on the
reference probability space $(\Omega,\mathcal{G},P)$ introduced in Section \ref{intro}, a family $({\bf P}_{\rm c})_{{\rm c}\in [d]^\infty}$
of probability measures and an infinite sequence  ${\rm F}=({\rm T}_k)_{k\ge1}$ of independent random trees, such that for each
${\rm c}=(c_1,c_2,\dots)\in [d]^\infty$ and $k\ge1$, under ${\bf P}_{\rm c}$,  ${\rm T}_k$ is a branching tree, with progeny law $\nu$,
whose root has type $r({\rm T}_k)=c_k$. In particular, for any random time $\alpha:(\Omega,\mathcal{G})\rightarrow\mathbb{N}$,
the sequence $\{{\rm T}_1,\dots,{\rm T}_{\alpha}\}$ is an element of  $\mathcal{F}_d$.  The infinite sequence ${\rm F}$ will be called a $d$-type branching
forest with progeny law $\nu$.\\

Let us denote by ${\rm F}^{(i)}$ the subforest of type $i$ of ${\rm F}$, as it is defined in subsection \ref{space}.
From the properties of $\nu$, it follows that for each $i\in [d]$,
the subforest ${\rm F}^{(i)}$ is a.s.~infinite, so that we may define a $\mathbb{Z}^d$ valued infinite random sequence
$X^{(i)}=(X^{i,1},\dots,X^{i,d})$, for $i\in [d]$, in the same way as in (\ref{6245}), that is $X_0^{(i)}=0$ and
\begin{equation}\label{4577}
X_{n+1}^{i,j}-X_n^{i,j}=p_j({\rm u}_{n+1}^{(i)})\,,\;\;\mbox{if $i\neq j\;$ and}\;\;\;X_{n+1}^{i,i}-X_n^{i,i}=
p_i({\rm u}_{n+1}^{(i)})-1\,,\;\;\;n\ge0\,,
\end{equation}
where $({\rm u}^{(i)}_n)_{n\ge1}$ is the labeling of ${\rm F}^{(i)}$ in its breadth first search order.
\begin{thm}\label{Lukasiewicz-Harris}
Let ${\rm F}$ be a $d$-type branching forest with progeny law $\nu$.
\begin{itemize}
\item[$1.$] Then for any ${\bf\rm c}=(c_1,c_2\dots)\in [d]^\infty$, under ${\bf P}_{\rm c}$, the chains
\begin{equation}\label{3577}
X^{(i)}=(X^{i,1},\dots,X^{i,d})\,,\;\;\;i=1,\dots,d
\end{equation}
are independent
random walks with step distribution
\[{\bf P}_{\rm c}\Big(X^{(i)}_1=(q_1,\dots,q_d)\Big)=\nu_i(q_1,\dots,q_{i-1},q_i+1,q_{i+1},\dots,q_d)\,.\]
In particular, their laws do not depend on ${\bf\rm c}$. For each $i\in [d]$, $X^{i,i}$ is a downward skip free random walk such
that $\liminf_{n\rightarrow+\infty}X_n^{i,i}=-\infty$, a.s.~and for $j\neq i$, $X^{i,j}$ is a renewal process.
\item[$2.$] For all integer $r\ge1$, almost surely there is ${\rm r}\in \mathbb{Z}_+^d$, with
$r=r_1+\dots+r_d$ and  such that there is a smallest solution ${\rm n}$ to the system $({\rm r},X)$.
\item[$3.$] Conversely, let $Y$ be a copy of $X$ and ${\bf\rm c}=(c_1,c_2\dots)\in [d]^\infty$. Then to $Y$ and ${\bf\rm c}$,
we may associate a unique $d$-type branching forest, with progeny law $\nu$ and root type sequence ${\rm c}$, whose coding random
walk is $Y$.
\end{itemize}
\end{thm}
\begin{proof} Part $1.$ just follows from the construction (\ref{4577}) of $X$. Since the order on the subforests ${\rm F}^{(i)}$
does not depend on the particular topology of ${\rm F}$,  from the branching property, it is clear that the chains $X^{(i)}$, $i\in [d]$
are  independent random walks. Then the expression of the law of $X^{(i)}$ is a direct consequence of $(\ref{4577})$. Recall that
$X^{i,i}$ is the Lukasiewicz-Harris path of the subforest ${\rm F}^{(i)}$, see section 6.2 of \cite{pi}. Moreover, since from the properties
of $\nu$, each subforest ${\rm F}^{(i)}$ is a.s.~infinite, the random walk $X^{i,i}$ satisfies
$\liminf_{n\rightarrow\infty}X^{i,i}_n=-\infty$, a.s. The fact that $X^{i,j}$, for $i\neq j$ is a renewal process is obvious.\\

Then part 2 is a direct consequence of the construction of $X$ and Theorem \ref{coding}. Let $r\ge1$ and first assume that the finite
forest $\{{\rm T}_1,\dots,{\rm T}_r\}$ consisting in the first $r$ trees of F contains at least one vertex of each type. Let ${\rm r}$
be the unique element of $\mathbb{Z}^d$ such that $\{{\rm T}_1,\dots,{\rm T}_r\}\in\mathcal{F}_d^{\rm r}$. Then,
by coding the forest $\{{\rm T}_1,\dots,{\rm T}_r\}$ and by applying Theorem \ref{coding}, we obtain that there is
${\rm n}\in\mathbb{N}^d$ which is the smallest solution of the system $({\rm r},X)$. Now if for instance $n\in [d-1]$ types are
missing in the first $r$ trees of F, then we can apply the same arguments by replacing $d$ by $d-n$ in Theorem \ref{coding}.

Then part 3.~is a consequence of part 2. For each $r\ge1$, we may associate a unique forest to $Y$ with $r$ trees.
 Since $r$ can be arbitrarily large, the result is proved.$\;\;\Box$\\
\end{proof}

In the same spirit as in \cite{mi}, the Lukasiewicz-Harris type coding that is displayed in Theorem \ref{Lukasiewicz-Harris} might be used to
obtain invariance principles, for any functional that can be encoded simply enough. Besides this result should provide a way to obtain a proper definition
of continuous multitype branching trees and forests. Actually, it is natural to think that the latter objects are coded by $d$ independent, $d$-dimensional
L\'evy processes, with $d-1$ increasing coordinates and a spectrally positive coordinate.\\

Now we are going to apply our coding of multitype branching forests to the law of their total population and give a proof of Theorem \ref{main}.
To that aim, we first need to establish the crucial combinatorial Lemma \ref{crucial}. Let $E$ be $\mathbb{Z}_+$ or a finite integer interval of the type
$\{0,1,\dots,m\}$, with $m\ge1$ and let $g:E\rightarrow \mathbb{Z}^d$, be any application such that $g(0)=0$. For $n\in E$ such that $n\ge1$,
the $n$-cyclical permutations of $g$ are the $n$ applications $g_{q,n}$, $q=0,\dots,n-1$ which are defined on $E$ by:
\begin{equation}\label{def3}
g_{q,n}(h)\stackrel{\mbox{\tiny(def)}}{=}
\left\{\begin{array}{ll}g(q+h)-g(q)\;\;\;\;\;\;&\mbox{if}\;\;0\le h\le n-q\,,\\
g(h-(n-q))+g(n)-g(q)&\mbox{if}\;\;n-q\leq h\leq n\\
g(h)&\mbox{if}\;\;h\geq n\,.
\end{array}\right.
\end{equation}
Note that $g_{0,n}\equiv g$.
The transformation $g\mapsto g_{q,n}$ consists in inverting the parts $\{g(h),\,0\leq h\leq q\}$ and
$\{g(h),\,q\leq h\leq n\}$ in such a way that the new application, $g_{q,n}$, has the same values as $g$ at
0 and $n$, i.e. $g_{q,n}(0)=0$ and $g_{q,n}(n)=g(n)$.\\

Let $x\in S_d$, with finite length ${\rm n}=(n_1,\dots,n_d)\in\mathbb{N}^d$ and recall the notation $x^{(i)}=(x^{i,1},\dots,x^{i,d})$
from  Definition \ref{9076}. Then we define the ${\rm n}$-cyclical permutations of $x$ by
\begin{equation}\label{2956}
x_{{\rm q},{\rm n}}:=(x^{(1)}_{q_1,n_1},\dots,x^{(d)}_{q_d,n_d})\,,\;\;\mbox{for all ${\rm q}=(q_1,\dots,q_d)$ such that $0\le{\rm q}\le {\rm n}-1_d$,}
\end{equation}
where we have set $1_d=(1,1,\dots,1)$. Each sequence $x_{{\rm q},{\rm n}}$ will simply be called a cyclical permutation of
$x$. Note that there are $n_1n_2\dots n_d$, cyclical permutations of $x$.  Let  ${\rm r}=(r_1,\dots,r_d)\in\mathbb{Z}^d_+$ be such that
${\rm r}>0$ and assume that ${\rm n}$ is a solution of the
system $({\rm r},x)$. Then note that ${\rm n}$ is also a solution of the system $({\rm r},x_{{\rm q},{\rm n}})$, for all ${\rm q}$ such that
$0\le{\rm q}\le {\rm n}-1_d$, that is,
\[r_j+\sum_{i=1}^dx_{{\rm q},{\rm n}}^{i,j}(n_i)=0\,,\;\;j\in[d]\,.\]
This remark raises the question of the number of cyclical permutations $x_{{\rm q},{\rm n}}$ of $x$, such that ${\rm n}$
is the smallest solution of the system $({\rm r},x_{{\rm q},{\rm n}})$.

\begin{definition}\label{4265}  Let $x\in S_d$, with finite length ${\rm n}=(n_1,\dots,n_d)\in\mathbb{N}^d$.
Let ${\rm r}=(r_1,\dots,r_d)\in\mathbb{Z}^d_+$ be such that ${\rm r}>0$ and assume that ${\rm n}$ is a solution of the
system $({\rm r},x)$. For $0\le{\rm q}\le {\rm n}-1_d$, we say that $x_{{\rm q},{\rm n}}$ is
a good ${\rm n}$-cyclical permutation of $x$ with respect to ${\rm r}$, if ${\rm n}$ is the smallest solution of the system
$({\rm r},x_{{\rm q},{\rm n}})$, that is $x_{{\rm q},{\rm n}}\in S_d^{\rm r}$. When no confusion is possible, we will simply say
that $x_{{\rm q},{\rm n}}$ is a good cyclical permutation of $x$.
\end{definition}

\noindent This definition and the next lemma extend the following argument developed in \cite{ta} for the proof of the Ballot Theorem: For an integer valued sequence $(x_k,0\le k\le n)$ such that $\Delta x_k\ge -1$ and $x_0=0$, $x_n=-k$, there are exactly $k$ cyclical permutations $x_{q,n}$ such that $x_{q,n}$ first hits $-k$ at time $n$. Here is a generalisation of this result.

\begin{lemma}[Multivariate Cyclic Lemma]\label{crucial} Let $x\in S_d$, with length ${\rm n}=(n_1,\dots,n_d)\in\mathbb{N}^d$ and let
${\rm r}=(r_1,\dots,r_d)\in\mathbb{Z}^d_+$ be such that ${\rm r}>0$. Assume that ${\rm n}$ is a solution of the
system $({\rm r},x)$ such that $x^{i,i}(n_i)\neq0$, for all $i\in [d]$. Then the number of good cyclical permutations
of $x$ is $\det((-x^{i,j}(n_i))_{i,j\in [d]})$.
\end{lemma}

\noindent This lemma will be proved in Section \ref{annex}. It is the essential argument  for the proof of the following extension of the Ballot
Theorem.

\begin{thm}[Multivariate Ballot Theorem]\label{ballot}
Let $Y=(Y^{(1)},\dots,Y^{(d)})$ be a stochastic process defined on $(\Omega,\mathcal{G},P)$,
with $Y^{(i)}=(Y^{i,j}(h), j\in [d],\,h\in\mathbb{Z}_+)$, $i\in [d]$ and $Y_0=0$.
We assume that the coordinates $Y^{i,j}$, for $i\neq j$ are $\mathbb{Z}_+$ valued, nondecreasing and that the coordinates $Y^{i,i}$
are $\mathbb{Z}$ valued and downward skip free. Fix  ${\rm n}=(n_1,\dots,n_d)\in\mathbb{N}^d$, then
we assume further that the process $Y$ is ${\rm n}$-cyclically exchangeable, that is for any
${\rm q}=(q_1,\dots,q_d)\in\mathbb{Z}_+^d$ such that ${\rm q}\le {\rm n}-1_d$,
\[
Y_{{\rm q},{\rm n}}\ed Y\,,
\]
where $Y_{{\rm q},{\rm n}}$ is defined as in $(\ref{2956})$ for deterministic functions.
Then for any ${\rm r}=(r_1,\dots,r_d)\in\mathbb{Z}_+^d$ such that ${\rm r}>0$
and $k_{ij}$, $i,j\in [d]$, such that $k_{ij}\in\mathbb{Z}_+$, for $i\neq j$ and $-k_{jj}=r_j+\sum_{i\neq j}k_{ij}$,
\begin{equation}
\begin{split}
&P\Big(Y^{i,j}_{n_i}=k_{ij},\,\mbox{$ i,j\in [d]$ and ${\rm n}$ is the smallest solution of $({\rm r},Y)$}\Big)\\
=&\frac{\mbox{\rm det}(-k_{ij})}{n_1n_2\dots n_d}P\Big(Y^{i,j}_{n_i}=k_{ij},\,i,j\in [d]\Big)\,.\label{5478}
\end{split}
\end{equation}
\end{thm}
\begin{proof} If $P(Y^{i,j}_{n_i}=k_{ij},i\neq j)=0$, then the result is clearly true. Suppose that it is not the case and
let $y=(y^{(1)},\dots,y^{(d)})$ be a deterministic function such that for all $ i,j\in [d]$, $y^{i,j}$ is defined on
$\mathbb{Z}_+$, $y^{i,j}(0)=0$,  $y^{i,j}(n_i)=k_{ij}$ and
\begin{equation}\label{5211}
P\Big((Y^{(i)}(h),0\le h\le n_i))=(y^{(i)}(h),0\le h\le n_i)\Big)>0\,.
\end{equation}
For ${\rm h}=(h_1,\dots,h_d)\in\mathbb{Z}_+^d$,  we set
$Y({\rm h}):=(Y^{(1)}(h_1),\dots,Y^{(d)}(h_d))$ and for $0\le {\rm q}\le {\rm n} -1_d$, using the notation of (\ref{2956}), we set
$y_{{\rm q},{\rm n}}({\rm h}):=(y_{q_1,n_1}^{(1)}(h_1),\dots,y_{q_d,n_d}^{(d)}(h_d))$.
Let us consider the set
\[
E_{y,{\rm n}}=\Big\{(y_{{\rm q},{\rm n}}({\rm h}),0\le{\rm h}\le {\rm n}):0\le {\rm q}\le {\rm n}-1\Big\}\,,
\]
of ${\rm n}$-cyclical permutations of $y$ over the multidimensional interval $[0,{\rm n}]$.
Then ${\rm Card}(E_{y,{\rm n}})=n_1n_2\dots n_d$ and since $Y=(Y^{(1)},\dots,Y^{(d)})$ is a cyclically exchangeable chain,
the law of $(Y({\rm h}),\,0\le {\rm h}\le {\rm n})$, conditionally to the set
$\{(Y({\rm h}),\,0\le {\rm h}\le {\rm n})\in E_{y,{\rm n}}\}$
is the uniform law in the set $E_{y,{\rm n}}$. Moreover, assume that $k_{ii}\neq0$ for all $i\in [d]$, then conditionally to the set
$\{(Y({\rm h}),\,0\le {\rm h}\le {\rm n})\in E_{y,{\rm n}}\}$, from Lemma \ref{crucial}, the number of good
cyclical permutations of $ (Y({\rm h}),\,0\le {\rm h}\le {\rm n})$ is $\det(-k_{ij})$.
Therefore,
\begin{eqnarray*}
&&P\Big(Y^{i,j}_{n_i}=k_{ij},\,\mbox{$ i,j\in [d]$ and ${\rm n}$ is the}\\
&&\qquad\mbox{smallest solution of $({\rm r},Y)$}\,|\,(Y({\rm h}),\,0\le {\rm h}\le {\rm n})\in E_{y,{\rm n}}\Big)
=\frac{\det(-k_{ij})}{n_1n_2\dots n_d}\,.\end{eqnarray*}
Then we obtain the result by summing the identity
\begin{eqnarray*}
&&P\Big(Y^{i,j}_{n_i}=k_{ij},\,\mbox{$i,j\in [d]$ and}\\
&&\quad\mbox{ ${\rm n}$ is the smallest solution of $({\rm r},Y)$},(Y({\rm h}),\,0\le {\rm h}\le {\rm n})\in E_{y,{\rm n}}\Big)\\
&=&\frac{\mbox{det}(-k_{ij})}{n_1n_2\dots n_d}P((Y({\rm h}),\,0\le {\rm h}\le {\rm n})\in E_{y,{\rm n}})\,,\end{eqnarray*}
over all functions $y$ satisfying (\ref{5211}) and with different sets $E_{y,{\rm n}}$ of cyclical permutations. Finally, if $k_{ii}=0$, for
some $i\in [d]$, then since $n_i\ge1$, we can see that both members of identity (\ref{5478}) are equal to 0.$\;\;\Box$\\
\end{proof}

\noindent{\bf Proof of Theorem ${\bf \ref{main}}$}:
Let ${\rm r}$, ${\rm n}$ and $k_{ij}$ be as in the statement. Let ${\rm F}$ be  a $d$-type branching forest with progeny law
$\nu$, as defined at the beginning of this section and such that the first $r$ trees have root type sequence $(c_1,\dots,c_r)\in C^{\rm r}_d$.
Let $X$  be the coding random walk of F, as defined in (\ref{4577}). Recall the notation of Theorem \ref{main}, then from the coding
of Subsection \ref{codingforests} and Theorem \ref{Lukasiewicz-Harris}, we may check that
\begin{eqnarray}
&&\p_{\rm r}\Big(O_1=n_1,\dots, O_d=n_d, A_{ij}=k_{ij}, i,j\in [d],i\neq j\Big)\label{5472}\\
&=&{\bf P}_{\rm c}\Big(X^{i,j}_{n_i}=k_{ij},\,\mbox{$ i,j\in [d]$ and ${\rm n}$ is the smallest solution of $({\rm r},X)$}\Big)\,.\nonumber
\end{eqnarray}
Assume first that ${\rm n}\in\mathbb{N}^d$, then since $X$ is clearly cyclically exchangeable in the sense of Theorem \ref{ballot},
we obtain by applying this theorem,
\begin{eqnarray*}
&&{\bf P}_{\rm c}\Big(X^{i,j}_{n_i}=k_{ij},\,\mbox{$i,j\in [d]$ and ${\rm n}$ is the smallest solution of $({\rm r},X)$}\Big)\\
&=&\frac{\mbox{det}(-k_{ij})}{n_1n_2\dots n_d}{\bf P}_{\rm c}\Big(X^{i,j}_{n_i}=k_{ij},\,i,j\in [d]\Big)\,.\end{eqnarray*}
On the other hand, since from Theorem \ref{Lukasiewicz-Harris}, the random walks $X^{(i)}$, $i\in [d]$ are independent, we have
\begin{eqnarray*}
&&{\bf P}_{\rm c}\Big(X^{i,j}_{n_i}=k_{ij},\,\mbox{$i,j\in [d]$ and ${\rm n}$ is the smallest solution of $({\rm r},X)$}\Big)\\
&=&\frac{\mbox{det}(-k_{ij})}{n_1n_2\dots n_d}\prod_{i=1}^dP\Big(X^{i,j}_{n_i}
=k_{ij},\, j\in [d]\Big)\,.
\end{eqnarray*}
Then from the expression of the law of $X$ given in Theorem \ref{Lukasiewicz-Harris}, we obtain
\begin{equation}
\begin{split}
&{\bf P}_{\rm c}\Big(X^{i,j}_{n_i}=k_{ij},\,\mbox{$i,j\in [d]$ and ${\rm n}$ is the smallest solution of $({\rm r},X)$}\Big)\\
=&\frac{\mbox{det}(-k_{ij})}{n_1n_2\dots n_d}\prod_{i=1}^d\nu_i^{*n_i}(k_{i1},\dots,k_{i(i-1)},n_i+k_{ii},k_{i(i+1)},\dots,k_{id})\,,\label{6483}
\end{split}
\end{equation}
and the result is proved in this case.

Now with no loss of generality, let us assume that for some $0<d'<d$, we have $n_1,\dots,n_{d'}\in\mathbb{N}$ and  $n_{d'+1}=\dots=n_{d}=0$. We
point out that from the assumption $n_j\ge-k_{jj}=r_j+\sum_{i\neq j}k_{ij}$,  in this case we necessarily have $r_j=0$ and $k_{ij}=0$, for all $i\in [d]$
and $j=d'+1,\dots,d$. Then provided we also have $k_{ij}=0$, for all $i=d'+1,\dots,d$ and $j\in [d]$, $i\neq j$,
\begin{eqnarray*}
&&{\bf P}_{\rm c}\Big(X^{i,j}_{n_i}=k_{ij},\,d'+1\le i\le d,\,j\in [d]\Big)\nonumber\\
&=&\prod_{i=d'+1}^d\nu_i^{*n_i}(k_{i1},\dots,k_{i(i-1)},n_i+k_{ii},k_{i(i+1)},\dots,k_{id})=1\,.\label{6483}
\end{eqnarray*}
Define the chain $X$ restricted to $\mathbb{Z}^{d'}$, by  $X'=(X^{'(1)},\dots,X^{'(d')})$, where $X^{'(i)}=(X^{i,1},\dots,X^{i,d'})$, $i\in [d']$.
Set also ${\rm n}'=(n_1,\dots,n_{d'})$ and ${\rm r}'=(r_1,\dots,r_{d'})$. Then under our assumption on the integers $k_{ij}$, the following identity is
satisfied,
\begin{eqnarray*}
&&\Big\{X^{i,j}_{n_i}=k_{ij},\,\mbox{$i,j\in [d]$ and ${\rm n}$ is the smallest solution of $({\rm r},X)$}\Big\}\\
&=&\Big\{X'^{i,j}_{n_i}=k_{ij},\,\mbox{$i,j\in [d']$ and ${\rm n}'$ is the smallest solution of $({\rm r}',X')$ and}\\
&&X^{i,j}_{n_i}=0,1\le i\le d',\,d'+1\le j\le d\Big\}\,,
\end{eqnarray*}
so that identity (\ref{5472}) can be rewritten as,
\begin{eqnarray*}
&&\p_{\rm r}\Big(O_1=n_1,\dots, O_d=n_d, A_{ij}=k_{ij},  i,j\in [d],i\neq j\Big)\\
&=&{\bf P}_{\rm c}\Big(X'^{i,j}_{n_i}=k_{ij},\,\mbox{$i,j\in [d']$ and ${\rm n}'$ is the smallest solution of $({\rm r}',X')$ and}\\
&&X^{i,j}_{n_i}=0,1\le i\le d',\,d'+1\le j\le d\Big)\,.
\end{eqnarray*}
Moreover, conditionally on the set
\[
\Big\{X'^{i,j'}_{n_i}=k_{ij'},\, X^{i,j}_{n_i}=0,i, j'\in [d'],\,d'+1\le j\le d\Big\}\,,
\]
the chain $X'$ is cyclically exchangeable, so that we can conclude in the same way as above that
\begin{eqnarray*}
&&\p_{\rm r}\Big(O_1=n_1,\dots, O_d=n_d, A_{ij}=k_{ij}, i,j\in [d],i\neq j\Big)\\
&=&\frac{\mbox{det}((-k_{ij})_{1\le i,j\le d'})}{n_1n_2\dots n_{d'}}
\prod_{i=1}^{d'}\nu_i^{*n_i}(k_{i1},\dots,k_{i(i-1)},n_i+k_{ii},k_{i(i+1)},\dots,k_{id})\\
&=&\frac{\mbox{det}((-k_{ij})_{1\le i,j\le d'})}{n_1n_2\dots n_{d'}}
\prod_{i=1}^d\nu_i^{*n_i}(k_{i1},\dots,k_{i(i-1)},n_i+k_{ii},k_{i(i+1)},\dots,k_{id})\,.
\end{eqnarray*}
Finally, if $k_{ij}\neq0$, for some $i=d'+1,\dots,d$ and $j\in [d]$, then the first and the third members of the above equality are equal to 0.
So the proof is complete.
 $\;\;\Box$

\section{Proof of Lemmas \ref{preliminary} and \ref{crucial}}\label{annex}

\noindent {\bf Proof of Lemma \ref{preliminary}}. Assume that there is a solution ${\rm s}=(s_1,\dots,s_d)$ to the system
$({\rm r},x)$, that is: $r_j+\sum_{i=1}^dx^{i,j}(s_i)=0$, $j\in [d]$. Let us write this equation in the following form:
\[r_j+\sum_{i\ne j}x^{i,j}(s_i)+x^{j,j}(s_j)=0\,,\;\;\;j=1,\dots,d\,.\]
Then recall that for fixed $j$, when the $k_i$'s increase, the term $\sum_{i\neq j}x^{i,j}(k_i)$ increases and when $k_j$ increases, the term
$x^{j,j}(k_j)$ may decrease only by jumps of amplitude $-1$.\\

For ${\rm k}=(k_1,\dots,k_d)=0$, we have $r_j+\sum_{i\ne j}x^{i,j}(k_i)+x^{j,j}(k_j)=r_j$, $j\in [d]$. So for the left hand
side of the later equation to reach 0, each $k_i$ has to be at least $\tau^{(i)}(r_i)$, where $\tau^{(i)}$ has been defined in
(\ref{5170}). (In this proof, we found it more convenient to use the notation $\tau^{(i)}(k)$ for $\tau_k^{(i)}$.) Then either
\[r_j+\sum_{i\ne j}x^{i,j}(\tau^{(i)}(r_i))+x^{j,j}(\tau^{(j)}(r_j))=0\,,\;\;\;j=1,\dots,d\,,\]
or all of the terms $r_j+\sum_{i\ne j}x^{i,j}(\tau^{(i)}(r_i))+x^{j,j}(\tau^{(j)}(r_j))$,
$j\in [d]$ are greater or equal than 0, at least one of them being strictly greater than 0.\\

Then in the latter case, for $r_j+\sum_{i\ne j}x^{i,j}(k_i)+x^{j,j}(k_j)$ to attain 0, each of the $k_j$'s has to be at least
$\tau^{(j)}(r_j+\sum_{i\neq j}x^{i,j}(\tau^{(i)}(r_i)))$. This argument can be repeated until all of the terms
$r_j+\sum_{i\ne j}x^{i,j}(k_i)+x^{j,j}(k_j)$ attain 0. More specifically, set $v_1^{(j)}=r_j$ and for $n\ge1$,
\[v_{n+1}^{(j)}=r_j+\sum_{i\neq j}x^{i,j}(\tau^{(i)}(v_n^{(i)}))\,,\]
and set $k_j^{(n)}=\tau^{(j)}(v_n^{(j)})$. For $n\ge1$, either
\[r_j+\sum_{i\neq j}x^{i,j}(k_i^{(n)})+x^{j,j}(k_j^{(n)})=0\,,\;\;\;j=1,\dots,d\,,\]
or all of the terms $r_j+\sum_{i\neq j}x^{i,j}(k_i^{(n)})+x^{j,j}(k_j^{(n)})$, $j\in [d]$ are greater or equal than 0,
at least one of them being strictly greater than 0. In the later case, for all of the terms $r_j+\sum_{i\ne j}x^{i,j}(k_i)+x^{j,j}(k_j)$,
$j\in [d]$ to vanish, the index ${\rm k}$ has to be at least ${\rm k}^{(n+1)}=(k_1^{(n+1)},\dots,k_d^{(n+1)})$.
But since there is a solution ${\rm s}$ to the equation $({\rm r},x)$, there is necessarily a finite index $n_0$ such that
${\rm k}^{(n_0)}\le {\rm s}$ and
\[r_j+\sum_{i\ne j}x^{i,j}(k_i^{(n_0)})+x^{j,j}(k_j^{(n_0)})=0\,.\]
That is, for all $j\in [d]$, $k_j^{(n_0)}=\tau^{(j)}(r_j+\sum_{i\ne j}x^{i,j}(k_i^{(n_0)}))$.
Hence ${\rm k}^{(n_0)}$ is the smallest solution of the system $({\rm r},x)$. Moreover by definition, $k_i^{(n_0)}=\min\{n:x^{i,i}_n=
\min_{0\le k\le k_i^{(n_0)}}x^{i,i}_k\}$. This proves the first part of the lemma.\\

Let ${\rm r}'=(r_1',\dots,r_d')\in\mathbb{Z}_+^d$ be such that ${\rm r}'\le{\rm r}$. Then we prove that there is a smallest
solution to the system $({\rm r}',x)$ similarly. More specifically, since there is a smallest solution to the
equation
\[r_j'+\sum_{i\ne j}x^{i,j}(s_i)+x^{j,j}(s_j)=r_j'-r_j\le 0\,,\;\;\;j=1,\dots,d\,,\]
then by the same arguments as in the first part, we prove that there is a smallest solution to the equation
\[r_j'+\sum_{i\ne j}x^{i,j}(s_i)+x^{j,j}(s_j)=0\,,\;\;\;j=1,\dots,d\,.\]
Let ${\rm k}$ and ${\rm k}'$ be respectively the smallest solutions of $({\rm r},x)$ and $({\rm r}',x)$.
Then we have,
\[r_j-r'_j+r'_j+\sum_{i=1}^dx^{i,j}(k'_i)+\sum_{i=1}^d\tilde{x}^{i,j}(k_i-k'_i)=0\,.\]
Since $r'_j+\sum_{i=1}^dx^{i,j}(k'_i)=0$, the above equation shows that ${\rm k}-{\rm k}'$ is a solution of the system
$({\rm r}-{\rm r}',\tilde{x})$. Moreover, if ${\rm k}''$ was a strictly smaller solution of $({\rm r}-{\rm r}',\tilde{x})$,
than ${\rm k}-{\rm k}'$ (i.e. ${\rm k}''<{\rm k}-{\rm k}'$) then we would have
\begin{eqnarray*}
r_j-r'_j+r'_j+\sum_{i=1}^dx^{i,j}(k'_i)+\sum_{i=1}^d\tilde{x}^{i,j}(k''_i)&=&r_j+\sum_{i=1}^dx^{i,j}(k'_i+k''_i)\\
&=&0\,,
\end{eqnarray*}
so that ${\rm k}'+{\rm k}''$ would be a solution of $({\rm r},x)$, strictly smaller than ${\rm k}$, which is
a contradiction.~\qed

\noindent {\bf Proof of Lemma \ref{crucial}}.
Recall from Lemma \ref{numbersubtrees} that to each forest $({\bf f},c_{\bf f})\in\bar{\mathcal{F}}_d^{\rm r}$, we can associate a
 system $({\rm r},x)$, with smallest solution ${\rm k}$, where $x\in\bar{S}_d^{\rm r}$. The matrix $(k_{ij})=(x^{i,j}(k_i))$ of the last values of $x$ will be called the {\it LV-matrix} of $x$ or the
LV-matrix of the forest $({\bf f},c_{\bf f})$.
Recall that $k_j=-k_{jj}=r_j+\sum_{i\neq j}k_{ij}$ and that for $i\neq j$, $k_{ij}$ is the total number of vertices in ${\rm\bf v}({\rm\bf f})$
with type $j$, whose parent has type $i$ and that $k_i=-k_{ii}$ is the total number of vertices of type $i$ in ${\rm\bf v}({\rm\bf f})$.\\

For the next definition, we say that the integers $a_1,\dots,a_n$ are ranked {\it in the increasing order, up to a cyclical permutation}, if there is a
cyclical permutation $\sigma$ of the set $[n]$, such that $a_{\sigma(1)},\dots,a_{\sigma(n)}$ are ranked in the increasing order.

\newpage

\begin{definition}\label{simple} Let ${\rm r}\in\mathbb{Z}_+^d$ be such that ${\rm r}>0$.
\begin{itemize}
\item[$1.$] An element $({\bf f},c_{\bf f})$ of  $\bar{\mathcal{F}}_d^{\rm r}$ is said to be a {\em simple forest} if for
each type $i\in [d]$, at most one vertex of type $i$ in ${\bf v}({\bf f})$ has children $($the others are leaves$)$ and if its vertices are labelled in the following way: To each vertex of type
$i$, we associate a different integer in $[k_i]$ which is called its {\em label}. For each $i$, the sequence of appearance of labels of vertices of type $i$ in the
breadth first search of $({\bf f},c_{\bf f})$ is  ranked in the increasing order, up to a cyclical permutation, see Figure~$\ref{eforest}$.

\item[$2.$] For  $x\in\bar{S}_d^{\rm r}$, the system $({\rm r},x)$, with smallest solution ${\rm k}=(k_1,\dots,k_d)$  is called a
{\em simple system}  if for all $i\in [d]$, there is $k_i'\le k_i-1$, such that for all $k=0,\dots, k_i'$ and $i\neq j$, $x^{i,j}(k)=0$
and for all $k$ such that $k_i'+1\le k\le k_i$ and $i\neq j$, $x^{i,j}(k)=x^{i,j}(k_i)$.
\end{itemize}
\end{definition}

\noindent In other words, simple systems are elements $x\in\bar{S}_d^{\rm r}$ such that for all $i\in [d]$, the sequences
$x^{i,j}$, $ j\in [d]$ have at most one positive jump that occurs at the same point in $\{0,\dots,k_i\}$.
Then we have the following straightforward result.

\begin{proposition}\label{1239} If a forest is simple, then its associated system
$($as defined in Lemma $\ref{numbersubtrees})$ is simple.
\end{proposition}
\begin{proof} Recall the notation of the beginning of Subsection \ref{codingforests}. If $({\bf f},c_{\bf f})$ is a simple forest, then for each $i$, there is at most one index $k=0,\dots,k_i-1$ such that $x_{k+1}^{i,j}-x_k^{i,j}=p_j(u_{k+1}^{(i)})>0$, for some $j\in[d]$, $(j\neq i)$. That is for all $k'\neq k$ and $j\in[d]$, $x_{k'+1}^{i,j}-x_{k'}^{i,j}=0$. \qed
\end{proof}

\begin{definition}\label{8348} An {\em elementary} forest is a forest of $\mathcal{F}_d$ that contains exactly one vertex of each type.
In particular, each elementary forest contains exactly $d$ vertices and is coded by the $d$ couples
$(j_i,i)$, $i\in [d]$, where $j_i$ is the type of the parent of the vertex of type $i$. If the vertex of type $i$ is a root,
then we set $j_i=0$.
We define the set $D$ of vectors $(j_1,\dots,j_d)$, $0\le j_i\le d$ such that $(j_i,i)$, $i\in [d]$ codes an elementary forest.
\end{definition}

\begin{lemma}\label{crucial0} Let ${\rm r}=(r_1,\dots,r_d)\in\mathbb{Z}_+^d$ such that ${\rm r}>0$ and let $({\rm r},x)$ be a simple
system with smallest solution ${\bf\rm k}=(k_1,\dots,k_d)\in\mathbb{N}^d$, then the number of good cyclical permutations
of $x$ is
\begin{equation}\label{1384}
\sum_{(j_1,\dots,j_d)\in D}\prod_{i=1}^d k_{j_i i}\,,
\end{equation}
where for $i,j\in [d]$, $k_{ij}=x^{i,j}(k_i)$ and where we have set $k_{0i}=r_i$.
\end{lemma}
\begin{proof} The proof will be performed by reasoning on forests.
Let $({\rm r},x)$ be as in the statement. From Lemma \ref{numbersubtrees} we can associate to $({\rm r},x)$ a forest $({\bf f},c_{\bf f})$ of
$\bar{\mathcal{F}}_d^{\rm r}$ such that for each type $i\in [d]$, at most one vertex of type $i$ in ${\bf v}({\bf f})$ has children $($the others are leaves$)$.
Then by labeling all the vertices of this forest in the breadth first search order, we obtain a simple forest, see Definition \ref{simple}.
Note that for each good cyclical permutation $x_{{\rm q},{\rm k}}$ of $x$, the system $({\rm r},x_{{\rm q},{\rm k}})$ is simple itself and to each one, we may
associate a unique simple forest which is obtain by cyclical permutations of the vertices of type $i$ in $({\bf f},c_{\bf f})$, for each $i$.
Conversely, recall  Proposition \ref{1239}, then the simple system $({\rm r},y)$
which is associated to each simple forest with LV-matrix $(k_{ij})$ through Lemma \ref{numbersubtrees} is necessarily obtained from a good
cyclical permutation of $x$. Indeed the corresponding sequences $y^{i,j}$, for $i\neq j$, have exactly one jump and are such that $y^{i,j}(k_i)=k_{ij}$.
So $y$ is nothing but a cyclical permutation of $x$. These arguments prove that the number of good cyclical permutations of $x$ is equal to the number
of simple forests with LV-matrix $(k_{ij})$.

Then we are going to prove that the number of simple forests with LV-matrix $(k_{ij})$ is $\sum_{(j_1,\dots,j_d)\in D}\prod_{i=1}^d k_{j_i i}$. We make the
additional assumption that for each $i$, there is exactly one vertex who has children. Then observe that to each simple forest, we can associate a unique elementary
forest in the following way: the vertex of type $j$ in the elementary forest is the parent of
the vertex of type $i$ if, in the simple forest the parent of the vertex of type $i$ that has children has type $j$ (recall
that $j=0$ if the vertex of type $i$ is a root). An example of an  elementary forest associated to a simple forest is given in Figure \ref{eforest}.
\begin{figure}[hbtp]

\vspace*{-6.5cm}

{\includegraphics[height=420pt,width=500pt]{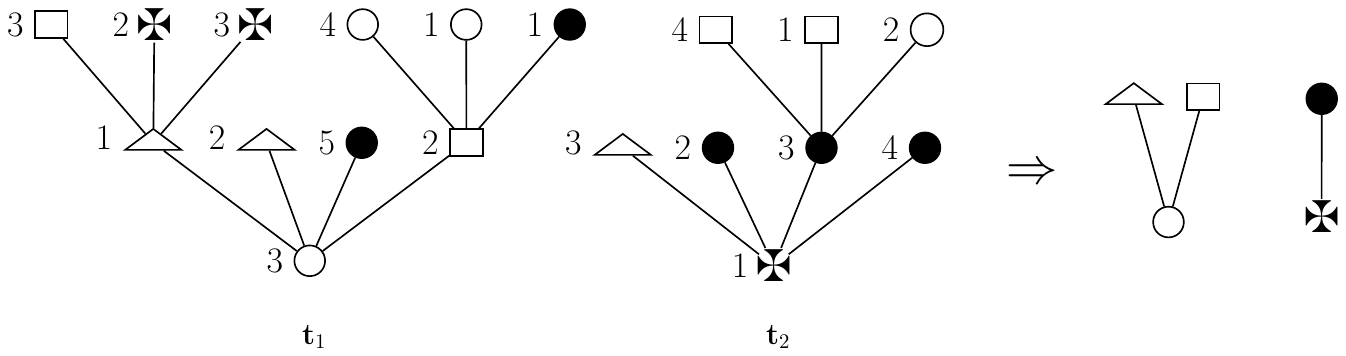}}

\vspace{-3cm}

\caption{A simple forest and its associated elementary forest.}\label{eforest}
\end{figure}
Then let $(j_1,\dots,j_d)\in D$. We easily see that the monomial $\prod_{i=1}^d k_{j_i i}$ is the number
of simple forests such that for each $i$, the parent of the vertex of type $i$ who has children has type $j_i$. Indeed, there are $k_{j_ii}$
possibilities to choose the vertex of type $i$ that has children. In other words, $\prod_{i=1}^d k_{j_i i}$ is the number
of all possible simple forests to which we can associate the same elementary forest
which is coded by $(j_i,i)$, $i\in [d]$. Then in order to obtain the total number of simple forests with LV-matrix $(k_{ij})$, it remains
to perform the summation of these monomials over all the possible elementary forests. So we obtained the formula of the statement, under our additional assumption.

Then we have proved the result for simple systems such that for all $i\in[d]$, there is $j\neq i$ with $k_{ij}>0$. Now assume that for all $i\in [d-1]$, there is $j\neq i$
such that $k_{ij}>0$ and that $k_{dj}=0$, for all $j\in [d]$. Then in such a system, we have $x^{(d)}\equiv 0$. Let us consider the system $({\rm r}',x')$, where
${\rm r}'=(r_1,\dots,r_{d-1})$ and $x'=(x^{(1)},\dots,x^{(d-1)})$. From what has just been proved, the number of good cyclical permutation of $({\rm r}',x')$ is $h_{d-1}:=\sum_{(j_1,\dots,j_{d-1})\in D'}\prod_{i=1}^{d-1} k_{j_i i}$ where $D'$ is the set that is defined in Definition \ref{8348} and where we replaced $d$ by
$d-1$.  Then in order to obtain all the good cyclical permutations of $x$ it remains to consider the $k_d$ cyclical permutations of the sequence $x^{(d)}$. Since the
latter are all identical, we have actually $k_d\times h_{d-1}$ good cyclical permutations
of $x$. Then we conclude by noticing that $k_d\times h_{d-1}=\sum_{(j_1,\dots,j_{d})\in D}\prod_{i=1}^{d} k_{j_i i}$, since
$k_d=r_d+\sum_{i\neq d}k_{id}=\sum_{i=0,\,i\neq d}^dk_{id}$. The general case where $k_{ij}=0$, for all $j\in [d]$, for other types  $i$ is obtained in the same
manner. \qed
\end{proof}

\noindent It is known that expression (\ref{1384}) is a determinant, see for instance the remark after Proposition 7 in \cite{bm}. This determinant is specified in the following lemma. The proof
which is given here essentially aims at quoting some references from which this result can be derived.
\begin{lemma}\label{crucial1}
For any ${\rm r}\in\mathbb{Z}_+^d$ and any integer valued matrix $(k_{ij})_{i,j\in [d]}$, such that $k_{ij}$ is nonnegative for
$i\neq j$ and  $-k_{jj}=r_j+\sum_{i\neq j}k_{ij}$, $j\in [d]$,
\[
\mbox{\rm det}(-k_{ij})=\sum_{(j_1,\dots,j_d)\in D}\prod_{i=1}^d k_{j_i i}\,,
\]
where the set $D$ is defined in Definition $\ref{8348}$ and $k_{0i}=r_i$.
\end{lemma}
\begin{proof} The proof is a direct consequence of the matrix tree theorem for directed graphs, due to Tutte \cite{tu}, see Section 3.6, page 470 therein.
However, Theorem 3.1 in \cite{moo} that implies Tutte's theorem is actually easier to apply, since it uses a setting which is closer to ours. Let us consider a set
$\{v_0,v_1,\dots,v_d\}$ of $d+1$ vertices and in the notations of \cite{moo}, set $W=L=\{v_0\}$. Then the family $\mathcal{F}_{W,L}$ that is described
in Theorem 3.1 of \cite{moo} is in bijection with the set of elementary forests, or equivalently with the set $D$, and identity $(3.2)$ in this theorem is
exactly $\mbox{\rm det}(-k_{ij})=\sum_{(j_1,\dots,j_d)\in D}\prod_{i=1}^d k_{j_i i}$.\qed
\end{proof}

\begin{lemma}\label{4311}
Let $x\in S_d$, with length ${\rm k}=(k_1,\dots,k_d)\in\mathbb{N}^d$. Let ${\rm r}=(r_1,\dots,r_d)\in\mathbb{Z}^d_+$ be such that
${\rm r}>0$ and assume that ${\rm k}$ is a solution of the system $({\rm r},x)$. Assume moreover that $x^{i,i}_k=-k$, for all
$k=0,\dots,k_i$ and $i\in [d]$. Then the number of good cyclical permutations of $x$ is $\mbox{\rm det}((x^{i,j}(k_i))_{i,j\in [d]})$.
\end{lemma}
\begin{proof} Let $x$ be as in the statement. If for all $i\in [d]$, the sequences $x^{i,j}$, $ j\in [d]$ have at most one positive jump that occurs at the same
point in $\{0,\dots,k_i\}$, then from Definition \ref{simple}, Lemma \ref{crucial0} and Lemma \ref{crucial1}, there is a cyclical permutation $x'$ of $x$ such that
$({\rm r},x')$ is a simple system and the result follows in this case.  In general, let us prove that there is a simple system, with LV-matrix $(x^{i,j}(k_i))$,
and whose number of good cyclical permutations is the same as  this of $x$.

Fix any index $m\in [d]$ and assume without loss of generality, that $m\neq 1$, and
$$
x^{m,1}(k_m)-x^{m,1}(k_m-1)>0.
$$
From $x$, we define a new sequence $\tilde{x}\in S_d$ with
length ${\rm k}$ as follows:
\begin{equation}\label{trans}
\left\{\begin{array}{l}
\tilde{x}^{m,j}(k)=x^{m,j}(k)\,,\;\;j=2,\dots,d,\;k=1,\dots,k_m\,,\\
\tilde{x}^{m,1}(k)=x^{m,1}(k)\,,\;\;\, k=1,\dots,k_m-2,k_m\,,\\
\tilde{x}^{m,1}(k_m-1)=x^{m,1}(k_m-1)+1.
\end{array}\right.
\end{equation}
All the other coordinates remain unchanged, that is,
\[\tilde{x}^{i,j}(k)=x^{i,j}(k)\,,\;\;i, j=1,\dots,d\,,\;\;k=1,\dots,k_i\,,\;\,\;\;i\neq m\,.\]
The sequence $\tilde{x}$ is obtained from $x$ by decreasing by one unit, the last jump of the coordinate $x^{m,1}$, that is
$x^{m,1}(k_m)-x^{m,1}(k_m-1)$. (Therefore the jump $x^{m,1}(k_m-1)-x^{m,1}(k_m-2)$ is increased by one unit.)
Denote by $N_{{\rm r},x}$ the number of good cyclical permutations of $x$.  We claim that
\begin{equation}\label{6248}
N_{{\rm r},\tilde{x}}\ge N_{{\rm r},x}\,.
\end{equation}
To achieve this aim, first observe that ${\rm k}$ is a solution of the system $({\rm r},\tilde{x})$, that is
\[r_j+\sum_{i=1}^d\tilde{x}^{i,j}(k_i)=0\,,\;\;j=1,\dots,d\,,\]
and note the straightforward inequality,
\[
\tilde{x}^{i,j}(k)\ge x^{i,j}(k)\,,\;\;i,j=1,\dots,d\,,\;\;k=1,\dots,k_i\,.
\]
Therefore, if ${\rm k}$ is the smallest solution of $({\rm r},x)$ then it is also the smallest solution of $({\rm r},\tilde{x})$.
Moreover, for ${\rm q}=(q_1,\dots,q_d)\in\mathbb{N}^d$ satisfying $q_m<k_{m}-1$, one easily checks that
the same inequality holds, that is
\[
\tilde{x}_{{\rm q},{\rm k}}^{i,j}(k)\ge x_{{\rm q},{\rm k}}^{i,j}(k)\,,\;\;i,j=1,\dots,d\,,\;\;k=1,\dots,k_i\,,
\]
so that if ${\rm k}$ is the smallest solution of $({\rm r},x_{{\rm q},{\rm k}})$ (i.e.  $x_{{\rm q},{\rm k}}$ is a good cyclical
permutation of $x$) then it is also the smallest solution of $({\rm r},\tilde{x}_{{\rm q},{\rm k}})$.

Now it remains to study the case where $q_m=k_m-1$. Assume that $x_{{\rm q},{\rm k}}$ is a good cyclical permutation of $x$, but that
$\tilde{x}_{{\rm q},{\rm k}}$ is not a good cyclical permutation of $\tilde{x}$. Then in order to obtain the inequality (\ref{6248}), we have to find
a good cyclical permutation  $\tilde{x}_{{\rm l},{\rm k}}$ of $\tilde{x}$ such that $x_{{\rm l},{\rm k}}$ is not a good cyclical permutation of $x$.
Let us define the sequence $\hat{x}\in S_d$ with length ${\rm k}$, which is obtained by decreasing by one unit the first coordinate of
$x^{(m)}_{{\rm q},{\rm k}}$, that is
\[\left\{\begin{array}{l}
\hat{x}^{(i)}\equiv x^{(i)}_{{\rm q},{\rm k}}\,,\ i\neq m\,,\\
\hat{x}^{(m)}(k)= x^{(m)}_{{\rm q},{\rm k}}(k)-{\bf e}_1,\,k\ge1\,,
\end{array}\right.\]
where ${\bf e}_1=(1,0,\cdots,0)$, is the $d$ dimensional unit vector.  Since
${\rm k}$ is the smallest solution of $({\rm r},x_{{\rm q},{\rm k}})$, then ${\rm k}$ is the smallest solution of $({\rm r}+{\bf e}_1,\hat{x})$
by the definition of $\hat{x}$. Moreover, from Lemma 2.2 $(ii)$, the system $({\bf e}_1,\hat{x})$ admits a smallest solution which is less than ${\rm k}$.
Let us call ${\rm p}$ this solution. Then ${\rm p}>0$ and from Lemma 2.2 (ii), ${\rm k}-{\rm p}$ is the smallest solution of
$({\rm r},\hat{x}_{{\rm p},{\rm k}})$.

Then let us consider the cyclical permutation of  $\tilde{x}_{{\rm q},{\rm k}}$ at ${\rm p}$. It is a cyclical permutation of $\tilde{x}$ that we shall
denote by $\tilde{x}_{{\rm l},{\rm k}}$. Note that $\hat{x}$ and $\tilde{x}_{{\rm q},{\rm k}}$ only differ from the last jump of $\hat{x}^{m,1}$
and  $\tilde{x}_{{\rm q},{\rm k}}^{m,1}$, more specifically, $\tilde{x}_{{\rm q},{\rm k}}(k_m)- \tilde{x}_{{\rm q},{\rm k}}(k_m-1)=
\hat{x}(k_m)- \hat{x}(k_m-1)+1$. Then from the above constructions, we can see that $\tilde{x}_{{\rm l},{\rm k}}$ is obtained as follows:
\[\tilde{x}_{{\rm l},{\rm k}}^{(i)}(k)=\left\{\begin{array}{ll}
\hat{x}_{{\rm p},{\rm k}}^{(i)}(k),\,&k\le k_i-p_i,\,i\neq m,\\
\hat{x}_{{\rm p},{\rm k}}^{(m)}(k),\,&k<k_m-p_m,\\
\hat{x}_{{\rm p},{\rm k}}^{(i)}(k_i-p_i)+\hat{x}^{(i)}(k-(k_i-p_i)),\,&k_i-p_i\le k\le k_i,\,i\neq m,\\
\hat{x}_{{\rm p},{\rm k}}^{(m)}(k_m-p_m)+\hat{x}^{(m)}(k-(k_m-p_m))+{\bf e}_1,\,&k_m-p_m\le k\le k_m\,.
\end{array}\right.\]
Since ${\rm k}-{\rm p}$ is the smallest solution of $({\rm r},\hat{x}_{{\rm p},{\rm k}})$ and since $\tilde{x}_{{\rm l},{\rm k}}$ is strictly
greater than $\hat{x}_{{\rm p},{\rm k}}$ at point ${\rm k}-{\rm p}$, there is no solution to $({\rm r},\tilde{x}_{{\rm l},{\rm k}})$, on the
(multidimensional) interval $[0,{\rm k}-{\rm p}]$. Moreover, also from the construction of $({\rm r},\tilde{x}_{{\rm l},{\rm k}})$,
since ${\rm p}$ is the smallest solution of $({\bf e}_1,\hat{x})$,  the only solution of $({\rm r},\tilde{x}_{{\rm l},{\rm k}})$ on the interval
$[{\rm k}-{\rm p},{\rm k}]$ is ${\rm k}$. Therefore, the smallest solution of $({\rm r},\tilde{x}_{{\rm l},{\rm k}})$ is ${\rm k}$.

On the other hand, note the following identity
\[
\Big(x_{{\rm l},{\rm k}}^{i,j}(k),\,i,j\in [d],\,k\in [k_i-p_i]\Big)=\Big(\hat{x}_{{\rm p},{\rm k}}^{i,j}(k),\,i,j\in [d],\,k\in [k_i-p_i]\Big)\,,
\]
which can be seen  directly from the definition of $\hat{x}$. It follows that ${\rm k}-{\rm p}$ is the smallest
solution of the system $({\rm r},x_{{\rm l},{\rm k}})$. But since ${\rm p}>0$, the sequence $x_{{\rm l},{\rm k}}$ is not a good permutation of
$x$, and the inequality (\ref{6248}) is proved.

Let ${\rm q}=(0,\dots,0,k_m-1,0\dots,0)$, where $k_m-1$ is the $m$-th coordinate of ${\rm q}$ and set $y:=\tilde{x}_{{\rm q},{\rm k}}$, then
by applying the same arguments as above to the chain $y$, we obtain that
\[N_{{\rm r},\tilde{y}}\ge N_{{\rm r},y}=N_{{\rm r},\tilde{x}}\ge N_{{\rm r},x}\,,\]
with obvious notations. But by reiterating $k_m$ times this operation, we obtain again the chain $x$. This shows that equality holds in (\ref{6248}),
that is $N_{{\rm r},\tilde{x}}= N_{{\rm r},x}$.

Finally, let $z\in S_d$ be a chain with length ${\rm k}$, such that $z^{i,j}(k_i)=x^{i,j}(k_i)$, for all $i,j\in [d]$, and such that for all $i\in [d]$, the sequences
$z^{i,j}$, $j\in [d]$ have at most one positive jump that occurs at the same point in $\{0,\dots,k_i\}$.  Then it is easy to see that the chain $x$ can be obtained
after several cyclical permutations and iterations of the transformation
$(\ref{trans})$ applied  to $z$, at any coordinate. Therefore  $N_{{\rm r},z}= N_{{\rm r},x}$. Assume that there is a good cyclical permutation $z'$
of $z$. Then $({\rm r},z')$ is a simple system. Therefore, from Lemma \ref{crucial0} and Lemma \ref{crucial1},
$N_{{\rm r},z'}=N_{{\rm r},z}= N_{{\rm r},x}=\mbox{det}(x^{i,j}(k_i))$.
\qed
\end{proof}

\begin{lemma}\label{4310}
Let $x\in S_d^{\rm r}$ and let $\bar{x}$ be the sequence of  $\bar{S}_d^{\rm r}$ which is associated to $x$, as in
$(\ref{5169})$. Then $x$ and $\bar{x}$ have the same number of good cyclical permutations.
\end{lemma}
\begin{proof}
Let ${\rm k}$ and ${\rm n}$ be the respective lengths of $\bar{x}$ and $x$. In particular, we have $n_i=\tau^{(i)}_{k_i}$,
$i\in [d]$. Let ${\rm q}\le {\rm k}-1_d$ be such that $\bar{x}_{{\rm q},{\rm k}}$ is a good cyclical permutation of $\bar{x}$. Then clearly,
there is ${\rm p}=(p_1,\dots,p_d)\le {\rm n}-1_d$, such that $p_i=\tau^{(i)}_{q_i}$. Set $y=x_{{\rm p},{\rm n}}$ and let us check that
\begin{equation}\label{5052}
\bar{y}=\bar{x}_{{\rm q},{\rm k}}\,.
\end{equation}
Define $\theta^{(i)}(m)=\min\{v:x_{p_i,n_i}^{i,i}(v)=-m\}=\min\{v:y^{i,i}(v)=-m\}$ and let $\tau_{q_i,k_i}^{(i)}$ be the
cyclical permutation of the sequence $\tau^{(i)}$, as defined in (\ref{def3}). Then from the construction of $y$, we can check that
\[\theta^{(i)}(m)=\tau_{q_i,k_i}^{(i)}(m)\,,\;\;m\le k_i\,,\]
from which we derive (\ref{5052}). Moreover, since $\bar{y}$ is a good cyclical permutation of $\bar{x}$, we deduce from Lemma
\ref{3379} that $y$ is a good cyclical permutation of $x$.

Conversely, let ${\rm q}<{\rm n}-1_d$ such that $x_{{\rm q},{\rm n}}$ is a good cyclical permutation of $x$. Then from part $(i)$
of Lemma \ref{preliminary}, we must have $n_i=\min\{n:x_{q_i,n_i}^{i,i}(n)=-k_i\}$. Therefore, there exists ${\rm p}<{\rm k}$,
such that $q_i=\tau^{(i)}_{p_i}$. Again, by setting $y=x_{{\rm q},{\rm n}}$, we check that
$\bar{y}=\bar{x}_{{\rm p},{\rm k}}$ and we deduce from Lemma \ref{3379} that $\bar{y}$ is a good cyclical permutation of $\bar{x}$.\qed
\end{proof}

\noindent Then we end the proof of Lemma \ref{crucial}. Let $x\in S_d$ be as in this lemma, that is $x$ has finite length
${\rm n}=(n_1,\dots,n_d)\in\mathbb{N}^d$ and ${\rm n}$ is a solution of the
system $({\rm r},x)$, where ${\rm r}=(r_1,\dots,r_d)\in\mathbb{Z}^d_+$ is such that ${\rm r}>0$. Let
$k_i'=-\min_{0\le n\le n_i}x^{i,i}_n$ and set $n'_i=\tau^{(i)}_{k'_i}$. Then ${\rm n}'\le {\rm n}$, so that we can define
the cyclical permutation $y=x_{{\rm n}',{\rm n}}$ of $x$. By construction, the sequence $\bar{y}$ has length ${\rm k}$,
where $k_i=-x^{i,i}(n_i)$ and $\bar{y}^{i,j}(k_i)=x^{i,j}(n_i)$, $i,j\in [d]$. Note that $k_i\ge1$, from the assumption $x^{i,i}(n_i)\neq0$.
Moreover, ${\rm k}$ is a solution of the system $({\rm r},\bar{y})$. So, thanks to Lemma \ref{4311}, the number of good cyclical permutations
of $\bar{y}$ is $\mbox{\rm det}(x^{i,j}(n_i))$ and from Lemmas \ref{3379} and \ref{4310} this is also the  number of good cyclical permutations
of $y$, the latter being clearly the number of good cyclical permutations of~$x$.~\qed

\vspace*{.5in}

\begin{singlespace} 
\end{singlespace}


\small
\begin{thebibliography}{99}
\bibitem{an} \sc K.B.~Athreya, P.~E.~Ney: \it Branching processes. \rm Die Grundlehren der mathematischen Wissenschaften,
Band 196. Springer-Verlag, New York-Heidelberg, 1972.
\bibitem{bm} \sc O.~Bernardi and A.H.~Morales: \rm Counting trees using symmetries.
{\it J. Combin. Theory} Ser. A, Vol 123, No 1, 104-122, (2014).
\bibitem{be} \sc J.~Bertoin: \rm The structure of the allelic partition of the total population for Galton-Watson processes with neutral
mutations. {\em Ann. Probab.} {\bf 37}(4), 1502--1523, (2009).
\bibitem{le} \sc J.F.~Le Gall: \rm Random trees and applications. {\it Probab. Surv.}, {\bf 2}, 245--311, (2005).
\bibitem{dw} \sc M.~Dwass: \rm The total population in a branching process and a related random walk. {\em J. Appl. Probab.}
{\bf 6}, 682--686, (1969).
\bibitem{go}\sc I.J.~Good: \rm Generalizations to several variables of Lagrange's expansion, with applications to stochastic processes.
{\it Proc. Cambridge Philos. Soc.} 56,  367--380, (1960).
\bibitem{ha1} \sc T.E.~Harris: \it The theory of branching processes. \rm Dover Phoenix Editions. Dover Publications, Inc., Mineola,
NY, 2002.
\bibitem{ha2} \sc T.E.~Harris: \rm First passage and recurrence distributions. {\it Trans. Amer. Math. Soc.}, {\bf 73}, 471--486, (1952).
\bibitem{ik} \sc I. Kortchemski: \rm Invariance principles for Galton-Watson trees
conditioned on the number of leaves, \emph {Stoch. Proc. Appl.},  \textbf{122}, 3126--3172, (2012).
\bibitem{mi} \sc G.~Miermont: \rm Invariance principles for spatial multitype Galton-Watson trees.
{\it Ann. Inst. Henri Poincar\'e Probab. Stat.}, {\bf 44}, no. 6, 1128--1161, (2008).
\bibitem{min} \sc N.~Minami: \rm On the number of vertices with a given degree in a Galton-Watson tree. {\it Adv.
Appl. Probab} {\bf 37}, 229--264 (2005).
\bibitem{mo} \sc C.J.~Mode: \it Multitype branching processes. Theory and applications. \rm
Modern Analytic and Computational Methods in Science and Mathematics, No. 34 American Elsevier Publishing Co., Inc., New York 1971.
\bibitem{moo} \sc J.W.~Moon: \rm Some determinant expansions and the matrix-tree theorem.
{\it Discrete Math.} 124, no. 1-3, 163--171, (1994).
\bibitem{ot} \sc R.~Otter: \rm  The multiplicative process. {\it Ann. Math. Statistics} 20, 206--224, (1949).
\bibitem{pi} \sc J.~Pitman: \it Combinatorial Stochastic Processes. \rm Saint-Flour, Springer, 2002.
\bibitem{ta} \sc L.~Tak\'acs: \rm A generalization of the ballot problem and its application in the theory of queues. {\it J. Amer. Statist. Assoc.} 57, 327--337, (1962).
\bibitem{tu} \sc W.T.~Tutte: \rm The dissection of equilateral triangles into equilateral triangles. {\it Proc. Cambridge Philos. Soc.} 44, 463--482, (1948).
\end{thebibliography}
\end{document}